\documentclass{amsart}
\usepackage{hyperref}
\usepackage{fullpage}
\usepackage{amsrefs}
\usepackage{tikz}

\usepackage[matrix,arrow]{xy}
\newdir{ >}{{}*!/-9pt/\dir{>}}

\newcommand{\CRings}	{\operatorname{CRings}}
\newcommand{\TFold}	{\operatorname{TFold}}
\newcommand{\Fold}	{\operatorname{Fold}}
\newcommand{\Graphs}	{\operatorname{Graphs}}
\newcommand{\NCM}       {\operatorname{NCM}}
\newcommand{\DNCM}      {\operatorname{DNCM}}
\newcommand{\Hom}       {\operatorname{Hom}}
\newcommand{\Map}       {\operatorname{Map}}

\renewcommand{\SS}	{\operatorname{SS}}

\newcommand{\dom}	{\operatorname{dom}}
\newcommand{\img}	{\operatorname{img}}
\newcommand{\edg}	{\operatorname{edge}}

\newcommand{\res}	{{\operatorname{res}}}
\newcommand{\rfl}	{{\operatorname{ref}}}
\newcommand{\rot}	{{\operatorname{rot}}}
\renewcommand{\vrt}	{\operatorname{vert}}

\newcommand{\al}        {\alpha}
\newcommand{\bt}        {\beta}
\newcommand{\dl}        {\delta}
\newcommand{\zt}        {\zeta}
\newcommand{\tht}       {\theta}
\newcommand{\ep}        {\epsilon}

\newcommand{\lm}        {\lambda}
\newcommand{\sg}        {\sigma}
\newcommand{\om}        {\omega}

\newcommand{\Sg}        {\Sigma}

\newcommand{\N}         {{\mathbb{N}}}
\newcommand{\Z}         {{\mathbb{Z}}}
\newcommand{\Q}         {{\mathbb{Q}}}
\newcommand{\R}         {{\mathbb{R}}}
\newcommand{\C}         {{\mathbb{C}}}

\newcommand{\sm}        {\setminus}
\newcommand{\sse}       {\subseteq}
\newcommand{\st}        {\;|\;}
\newcommand{\CP}        {{\mathbb{C}P}}
\newcommand{\ip}[1]     {\left\langle #1\right\rangle}
\newcommand{\ot}        {\otimes}
\newcommand{\ov}[1]     {\overline{#1}}
\newcommand{\tm}        {\times}
\newcommand{\un}[1]     {\underline{#1}}
\newcommand{\xra}       {\xrightarrow}
\newcommand{\tp}        {\widetilde{p}}
\newcommand{\uW}        {\underline{W}}

\newcommand{\B}	        {\mathbf{2}}

\newcommand{\tA}	{\widetilde{A}}
\newcommand{\tH}	{\widetilde{H}}
\newcommand{\tV}	{\widetilde{V}}
\newcommand{\tW}	{\widetilde{W}}

\newcommand{\bE}	{\overline{E}}
\newcommand{\bG}	{\overline{G}}
\newcommand{\bV}	{\overline{V}}
\newcommand{\be}	{\overline{e}}
\newcommand{\bom}       {\overline{\omega}}
\newcommand{\psb}[1]    {[\![#1]\!]}
\newcommand{\bcf}[2]{\left(\begin{array}{c}{#1}\\{#2}\end{array}\right)}

\renewcommand{\:}{\colon}
\renewcommand{\ss}{\scriptstyle}

\newtheorem{theorem}{Theorem}[section]

\newtheorem{lemma}[theorem]{Lemma}
\newtheorem{proposition}[theorem]{Proposition}
\newtheorem{corollary}[theorem]{Corollary}
\theoremstyle{definition}
\newtheorem{remark}[theorem]{Remark}
\newtheorem{definition}[theorem]{Definition}
\newtheorem{example}[theorem]{Example}
\newtheorem{construction}[theorem]{Construction}

\begin{document}
\title{Geometry and cohomology of Khovanov-Springer varieties}
\author{Philip Eve and Neil Strickland}

\maketitle

\section{Introduction}
\label{sec-intro}

Let $n$ be a natural number.  We can then consider the vector space
$V(n)=(\C[t]/t^n)^2$ as a module over $\C[t]$, and we define $X(n)$ to
be the set of complete flags
\[ 0 = W_0 < W_1 < \dotsb < W_{2n} = V(n) \]
for which each space $W_i$ is a submodule of $V(n)$.  This defines a
subvariety (which for $n > 1$ is not a manifold) of the usual variety of all
complete flags in $V(n)$.  We call $X(n)$ a \emph{Khovanov-Springer
 variety}.

The cohomology of $X(n)$ (and a more general class of spaces that we
shall not discuss) has been studied extensively using methods of Lie
theory, representation theory and \'etale cohomology of varieties in
characteristic $p$.  Much interest has been driven by the fact that
the cohomology groups have an action of the symmetric group $\Sg_{2n}$
(not arising from an action on the space itself), and the resulting
representations of $\Sg_{2n}$ are of independent interest.  The
paper~\cite{copr:sfc} by De Concini and Procesi is a good entry point
into this literature.  Section~4 of that paper gives a description of
$H^*X(n)$ by generators and relations, with a proof depending on
earlier work of Borho and Kraft~\cite{bokr:ubd}.  (In the notation
of~\cite{copr:sfc}, $X(n)$ is $\mathcal{F}_\eta$, where $\eta$ is the
partition of $2n$ into two blocks of size $n$.)  Dually, the homology
of $X(n)$ has been described by Russell and Tymoczko~\cite{ruty:srk}.
This determines the cohomology additively, but not multiplicatively.

Our main aim in the present paper to give a new proof of the ring
structure of $H^*X(n)$ using a very different set of methods.  This
will reveal some new combinatorial, geometric and algebraic structure
of the spaces $X(n)$.

\begin{definition}
 For $1\leq i\leq 2n$ we let $L_i$ denote the complex line bundle
 over $X(n)$ whose fibre at a flag $\un{W}$ is the quotient
 $W_i/W_{i-1}$.  We write $x_i$ for the Euler class of $L_i$, so
 $x_i\in H^2(X(n))$.  We write $\sg_k$ for the $k$'th elementary
 symmetric function of $x_1,\dotsc,x_{2n}$.
\end{definition}

\begin{definition}
 We say that a subset $J\sse\{1,\dotsc,2n\}$ is \emph{sparse} if for
 each $j\in J$ we have $|J^c_{>j}|>|J_{>j}|$, where
 \begin{align*}
  J_{>j}   &= \{ k\in J \st k>j \} \\
  J^c_{>j} &= \{ k\in \{1,\dotsc,2n\}\sm J\st k>j \}.
 \end{align*}
 That is: $J$ is sparse iff for each $j \in J$, a majority of the elements of
 $\{j + 1, \ldots, 2n\}$ do not belong to $J$. For any such set $J$, we put
 $x_J=\prod_{j\in J}x_j\in H^{2|J|}(X(n))$.
\end{definition}

\begin{theorem}\label{thm-main}
 We have
 \[ H^*(X(n)) =
     \frac{\Z[x_1,x_2,\dotsc,x_{2n}]}{
      (x_1^2,\dotsc,x_{2n}^2,\sg_1,\dotsc,\sg_{2n})
     }.
 \]
 Moreover, the set
 \[ BR(n) = \{ x_J \st J\sse\{1,\dotsc,2n\} \text{ is sparse } \} \]
 is a basis for $H^*(X(n))$ over $\Z$.
\end{theorem}
This will be restated and proved as Theorem~\ref{thm-main-bis}.  The
first half can also be obtained by specialising the results
of~\cite{copr:sfc}*{Section 4}.

\begin{remark}\label{rem-main-examples}
 Initial cases are as follows.  The space $X(0)$ is a single point,
 and $R(0)=\Z$.  The space $X(1)$ is just $\CP^1\simeq S^2$, and
 $R(1)=\Z[x_1]/x_1^2$.  It can be shown that there is a pushout square
 \[ \xymatrix{
   \CP^1 \ar@{ >->}[r]^\dl \ar@{ >->}[d]_\dl &
   \CP^1\tm\CP^1 \ar@{ >->}[d] \\
   \CP^1\tm\CP^1 \ar@{ >->}[r] & X(2)
 } \]
 (where $\dl$ is the diagonal map).  We also have
 \[ BR(2) = \{1,x_1,x_2,x_3,x_1x_2,x_1x_3\}.  \]
\end{remark}

We now offer some remarks on the method of proof.  First, the elements
$\sg_k$ are (up to sign) the Chern classes of the bundle
$\bigoplus_{i=1}^{2n}L_i$, but that bundle is easily seen to be
isomorphic to the constant bundle with fibre $V(n)$, so we see that
$\sg_k=0$.  Next, in Section~\ref{sec-squares} we will give an
explicit embedding of $L_i$ as a subbundle of the constant bundle with
fibre $\C^2$; from this it follows easily that $x_i^2=0$.  Thus, we at
least have a map $\phi\:R(n)\to H^*(X(n))$, where
\[ R(n) =
     \frac{\Z[x_1,x_2,\dotsc,x_{2n}]}{
      (x_1^2,\dotsc,x_{2n}^2,\sg_1,\dotsc,\sg_{2n})
     }.
\]

Next, an inductive argument based on the form of the relations shows
that the set $BR(n)$ generates $R(n)$ as an abelian group.  For any
sparse set $K$ of size $n$ we will define a ring map $\rho_K$ from
$R(n)$ to the ring
\[ E(K) = \Z[x_k\st k\in K]/(x_k^2\st k\in K). \]
Collecting these together we obtain a ring map
$\rho\:R(n)\to\prod_KE(K)$.  There is a natural ordered basis for
$\prod_KE(K)$, and the elements $\rho(x_J)$ (for $x_J\in BR(n)$) have
different leading terms with respect to this ordering.  Using this, we
see that $BR(n)$ is actually a basis for $R(n)$.

Next, we can construct subvarieties $X(n,K)\sse X(n)$ and isomorphisms
$H^*(X(n,K))=E(K)$ that are compatible with $\rho_K$ in an evident
sense.  Together these give a topological realisation of the map
$\rho$ and thus a proof that the map $\phi\:R(n)\to H^*(X(n))$ is
injective (and in fact a split monomorphism of abelian groups).

The real issue is now to prove that $\phi$ is surjective.  The most
efficient approach would be to quote the results of Russell and
Tymoczko that determine the homology of $X(n)$.  Using these and some
fairly straightforward combinatorics we could show that $R(n)$ has the
same total rank as $H^*(X(n))$, and the claim would follow.

However, we prefer to give an independent proof which reveals some
interesting additional structure along the way.  For this we need to
fit the subvarieties $X(n,K)$ into a more elaborate system of
subvarieties that we can use for inductive arguments.  To define and
study these subvarieties, we need some algebraic theory of modules
over $\C[t]/t^n$, and some combinatorial theory of bipartite graphs.
In particular, we will consider the graph $C(n)$ which is (the
boundary of) a $2n$-gon.  For any bipartite graph $G$ we will define a
space $Y(G)$ and a ring $S(G)$.  It will be clear that $S(C(n))=R(n)$,
and we will also be able to prove that $Y(C(n))=X(n)$.  Using this we
can define a map $S(G)\to H^*(Y(G))$ generalising our earlier map
$R(n)\to H^*(X(n))$.  We will prove that this is an isomorphism for a
large class of graphs including $C(n)$.  We do not know whether it is
an isomorphism for all $G$.

\section{Embedding line bundles}
\label{sec-squares}

\begin{definition}\label{defn-ip}
 We give $\C^2$ the usual Hermitian inner product
 \[ \ip{(u,v),(x,y)} = u\ov{x}+v\ov{y}. \]
 Any element of $a\in V(n)=(\C[t]/t^n)^2$ can be expressed uniquely as
 $a=\sum_{i=0}^{n-1}a_it^i$ with $a_i\in\C^2$.  We define a Hermitian
 inner product on $V(n)$ by the rule
 \[ \ip{\sum_{i=0}^{n-1}a_it^i,\sum_{i=0}^{n-1}b_it^i} =
     \sum_{i=0}^{n-1}\ip{a_i,b_i} \in\C.
 \]
 We also write
 \[ t^*\!\left(\,\sum_{i=0}^{n-1}a_it^i\!\right) =
     \sum_{i=1}^{n-1}a_it^{i-1},
 \]
 so $t^*\:V(n)\to V(n)$ is adjoint to multiplication by $t$.
\end{definition}

\begin{definition}\label{defn-om}
 We define a $\C$-linear map $\om\:V(n)\to\C^2$ by
 $\om\!\left(\sum_{i=0}^{n-1}a_it^i\right)=\sum_{i=0}^{n-1}a_i$.
\end{definition}

\begin{lemma}\label{lem-t-zero}
 For $\un{W}\in X(n)$ and $i>0$ we have $\dim(W_i)=i$ and
 $tW_i\leq W_{i-1}$.
\end{lemma}
\begin{proof}
 There are $2n + 1$ of the subspaces $W_i$, each a proper subspace of
 the next, and $\dim(W_0) = 0$, and $\dim(W_{2n}) = \dim(V(n)) = 2n$.
 From this it is clear that we must have $\dim(W_i) = i$ for each $i$.
 It follows that $W_i/W_{i-1}$ is a $\C[t]$-module with
 $\C$--dimension 1, with generator $u$ say.  We must then have $tu=zu$
 for some $z\in\C$, so $t^nu=z^nu$, but $t^n$ acts as zero on $V(n)$,
 so $z^n=0$, so $z=0$, so $tu=0$.  This means that $t$ acts as zero on
 $W_i/W_{i-1}$, or in other words $tW_i\leq W_{i-1}$.
\end{proof}

The following result is due to Cautis and Kamnitzer~\cite{caka:khd}.
\begin{proposition}\label{prop-isometry}
 Suppose we have $\C[t]$-submodules $T\leq U\leq V(n)$ with
 $tU\leq T$, and we put
 \[ P = U\ominus T = U\cap T^\perp. \]
 Then the restriction $\om\:P\to\C^2$ is an isometric embedding.
\end{proposition}
\begin{proof}
 Consider elements $a=\sum_{i=0}^{n-1}a_it^i$ and
 $b=\sum_{i=0}^{n-1}b_it^i$ in $P$.  By direct expansion we have
 \[ \ip{\om(a),\om(b)} = \sum_{i,j=0}^{n-1}\ip{a_i,b_j} =
     \ip{a,b} +
      \sum_{k=1}^{n-1}\left(\ip{t^ka,b}+\ip{a,t^kb}\right).
 \]
 On the right hand side, we have $t^ka\in t^kP\leq tU\leq T$ and
 $b\in P\leq T^\perp$ so $\ip{t^ka,b}=0$, and similarly
 $\ip{a,t^kb}=0$, so $\ip{\om(a),\om(b)}=\ip{a,b}$ as claimed.
\end{proof}

\begin{corollary}\label{cor-Li-embeds}
 There is a linear embedding $\om_i$ from the line bundle $L_i$ to the
 constant bundle with fibre $\C^2$.
\end{corollary}
\begin{proof}
 Consider a flag $\un{W}\in X(n)$.  The projection
 $\pi\:W_i\ominus W_{i-1}\to W_i/W_{i-1}=(L_i)_{\un{W}}$ is clearly an
 isomorphism, and we define $\om_i$ to be the composite
 $\om\circ\pi^{-1}$.  Using Lemma~\ref{lem-t-zero} and
 Proposition~\ref{prop-isometry} we see that this is injective.
\end{proof}

\begin{corollary}
 The map $\om_i$ induces a map $\bom_i\:X(n)\to\CP^1=S^2$ and an
 isomorphism $L_i\simeq\bom_i^*(T)$, where $T$ is the tautological
 bundle over $\CP^1$.
\end{corollary}
\begin{proof}
 Clear.
\end{proof}

\begin{corollary}\label{cor-squares-zero}
 In $H^*(X(n))$ we have $x_i^2=0$ for $i=1,\dotsc,2n$.
\end{corollary}
\begin{proof}
 Let $x\in H^2(\CP^1)$ denote the Euler class of $T$, so the previous
 corollary gives $x_i=\bom_i^*(x)$.  It is well-known that $x^2=0$, so
 $x_i^2=0$. 
\end{proof}

\begin{corollary}\label{cor-phi-exists}
 There is a unique homomorphism $\phi\:R(n)\to H^*(X(n))$ such that
 $\phi(x_i)=x_i$ for all $i$.
\end{corollary}
\begin{proof}
 We explained in the introduction that the elementary symmetric
 functions in the generators $x_i$ are zero.  We have now seen that we
 also have $x_i^2=0$, and the claim is immediate from this.
\end{proof}

\begin{remark}\label{rem-not-algebraic}
 The operation $(U,T)\mapsto U\ominus T$, when written in terms of
 local holomorphic coordinates, involves complex conjugation.  It
 follows that this is not a morphism of complex manifolds or of
 complex algebraic varieties.  Similarly, our map
 $\om_i\:X(n)\to\CP^1$ is not a morphism of algebraic varieties.
 Because of this, we do not know whether the relation $x_i^2=0$ lifts
 to the Chow ring of $X(n)$.
\end{remark}

\begin{proposition}\label{prop-om-emb}
 The maps $\bom_i$ combine to give a closed embedding
 $\bom\:X(n)\to\prod_{i=1}^{2n}\CP^1=(S^2)^{2n}$.
\end{proposition}
\begin{proof}
 If $\bom(\un{W})=\bom(\un{W}')$ then an evident induction
 based on Corollary~\ref{cor-Li-embeds} gives $W_i=W'_i$ for all $i$.
 This proves that $\bom$ is injective, and it is clearly also
 continuous.  Moreover, the target space $\prod_{i=1}^{2n}\CP^1$ is
 clearly compact and Hausdorff.  The same applies to $X(n)$, because
 is is a closed subvariety of a finite product of Grassmannians.  It
 follows that $\bom$ is automatically a closed embedding.
\end{proof}

\section{Symmetries of $X(n)$}
\label{sec-symmetry}

We next explain some automorphisms of the space $X(n)$.  Although we
will not make too much use of them, they will motivate some details of
our subsequent work, which are arranged to respect the symmetry as far
as possible.

\begin{definition}\label{defn-om-periodic}
 We previously defined maps $\bom_i\:X(n)\to\CP^1$ for
 $1\leq i\leq 2n$.  From now on we extend this definition to all
 integers by the rule $\bom_{i+2nk}=\bom_i$ for all $k\in\Z$.
\end{definition}

\begin{proposition}\label{prop-Xn-symmetry}
 There are unique maps $\rfl,\rot\:X(n)\to X(n)$ such that
 \begin{itemize}
  \item[(a)] $\bom_i\rot=\bom_{1+i}$ for all $i$
  \item[(b)] $\bom_i\rfl=\bom_{1-i}$ for all $i$
  \item[(c)] $\rot^{2n}=\rfl^2=(\rfl\rot)^2=1$.
 \end{itemize}
 Thus, these maps give an action of the dihedral group of order $4n$
 on $X(n)$.
\end{proposition}
\begin{proof}
 We will define $\rot$ in Definition~\ref{defn-rot}, and check
 claim~(a) in Proposition~\ref{prop-om-rot}.  Similarly, claim~(b) is
 covered by Definition~\ref{defn-ref} and
 Proposition~\ref{prop-om-ref}.  It may then be checked that for each $i$ we have
 \[ \bom_i\rot^{2n}=\bom_i\rfl^2=\bom_i(\rfl\rot)^2=\bom_i.
 \]
 This transfers to $\bom$, and as $\bom$ is an embedding we can conclude that~(c) holds.
\end{proof}

\begin{definition}\label{defn-unroll}
 We put
 \begin{align*}
  A &= \C[t] & \tA &= \C[t,t^{-1}] & A(n) &= A/t^n \\
  V &= A^2   & \tV &= \tA^2        & V(n) &= A(n)^2
 \end{align*}
 (so $V(n)\simeq\C^{2n}$ as before).  We define $\om\:\tV\to\C^2$ by
 $\om\!\left(\sum_iv_it^i\right)=\sum_iv_i$, and we define an inner product on
 $\tV$ by $\ip{\sum_iu_it^i,\sum_iv_it^i}=\sum_i\ip{u_i,v_i}$.
 We also write $\pi\:V\to V(n)=V/t^nV$ for the quotient map.

 Now suppose we have a flag $\uW\in X(n)$.  For $0\leq i\leq 2n$ we
 put
 \[ \tW_i = \{v\in t^{-n}V\st \pi(t^nv)\in W_i\}. \]
 In particular, we
 have $\tW_0=V$ and $\tW_{2n}=t^{-n}V$.  We then define $\tW_k$ for
 all $k\in\Z$ by $\tW_{2nj+i}=t^{-nj}\tW_i$.  We call this process
 \emph{unrolling}.  It produces a sequence of spaces
 $(\tW_k)_{k\in\Z}$ such that
 \begin{itemize}
  \item[(a)] $t\tW_k<\tW_{k-1}<\tW_k$ for all $k$, with
   $\dim(\tW_k/\tW_{k-1})=1$.
  \item[(b)] $\tW_{k+2n}=t^{-n}\tW_k$ for all $k$.
  \item[(c)] $\tW_0=V$.
 \end{itemize}
 We call such a sequence an \emph{unrolled flag}.  It is clear that
 unrolling gives a homeomorphism from $X(n)$ to the space of unrolled
 flags.
\end{definition}

\begin{remark}\label{rem-ominus}
 We will need to take complements of various subspaces $P\leq\tV$.
 This is potentially problematic because $\tV$ has infinite dimension
 and so the usual rules such as $P^{\perp\perp}=P$ and
 $\tV=P\oplus P^\perp$ need not hold for arbitrary subspaces.
 However, it is not hard to see that these rules do hold if
 $t^mV\leq P\leq t^{-m}V$ for some $m\geq 0$, because in that case we
 can reduce everything to a calculation in the finite-dimensional
 space $\bigoplus_{k=-m}^{m}t^k.\C^2$.  This observation, and minor
 variants, will cover all the spaces that we need.
\end{remark}
\begin{remark}\label{rem-om-periodic}
 For an unrolled flag $\tW$ and an integer $i$ we can define
 $\bom_i(\tW)=\om(\tW_i\ominus\tW_{i-1})$.  As multiplication by $t$
 preserves inner products (and thus orthogonal complements) and
 $\om(ta)=\om(a)$ we see that this is consistent with
 Definition~\ref{defn-om-periodic}.
\end{remark}

\begin{definition}\label{defn-rot}
 For any one-dimensional subspace $L<\C^2$ we have a splitting
 \[ \tV=(\tA\ot L)\oplus(\tA\ot L^\perp), \]
 and we define
 \[ \tht_L=(t.1_L)\oplus(1_{L^\perp})\:\tV \to \tV. \]
 We note that $\tht_L$ is $\tA$--linear and invertible, that
 it preserves inner products, and that $\om\tht_L=\om\:\tV\to\C^2$.

 Now consider an unrolled flag $\tW$.  Put
 $L=\om(\tW_1\ominus\tW_0)<\C^2$, and define $\rot(\tW)_i=\tht_L(\tW_{i+1})$
 for all $i\in\Z$.  After noting that
 \begin{align*}
  \tW_0 &= (A\ot L)\oplus(A\ot L^\perp) = V \\
  \tW_1 &= (t^{-1}A\ot L)\oplus(A\ot L^\perp),
 \end{align*}
 we see that $\rot(\tW)$ is again an unrolled flag.  We thus have a
 map $\rot\:X(n)\to X(n)$.
\end{definition}

\begin{proposition}\label{prop-om-rot}
 For any unrolled flag $\tW$ and $i\in\Z$ we have
 $\bom_i(\rot(\tW))=\bom_{i+1}(\tW)$.
\end{proposition}
\begin{proof}
 Since $\tht_L$ preserves inner products (and thus orthogonal complements)
 and $\om\tht_L=\om$, we see that
 \[ \bom_i(\rot(\tW)) =
     \om(\rot(\tW)_i \ominus \rot(\tW)_{i-1}) =
      \om(\tht_L(\tW_{i+1}) \ominus \tht_L(\tW_i)) =
       \om(\tW_{i+1} \ominus \tW_i) =
        \bom_{i+1}(\tW),
 \]
 as claimed.
\end{proof}

\begin{definition}\label{defn-ref}
 For any element $a=\sum_ia_it^i\in\tA$, we define
 $\xi(a)=\sum_ia_it^{-1-i}$.  This gives a map $\xi\:\tA\to\tA$ with
 $\xi^2=1$ and $\om\xi=\om$ and $\ip{\xi(a),\xi(b)}=\ip{a,b}$.

 Next, for any unrolled flag $\tW$ we define
 \[ \rfl(\tW)_i = \xi(\tW_{-i})^\perp. \]
\end{definition}

\begin{proposition}\label{prop-om-ref}
 This definition gives a map $\rfl\:X(n)\to X(n)$ with
 $\bom_i\rfl=\bom_{1-i}$.
\end{proposition}
\begin{proof}
 First, we have $\xi(V)=\bigoplus_{k<0}t^k.\C^2$ and so
 $\xi(V)^\perp=V$.  This shows that $\rfl(\tW)_0=V$ as required.  We
 also have
 \[ \rfl(\tW)_{i+2n} =
     \xi(\tW_{-i-2n})^\perp =
      \xi(t^n\,\tW_{-i})^\perp =
       t^{-n}\xi(\tW_{-i})^\perp =
        t^{-n}\rfl(\tW)_i.
 \]

 Next, we have $\tW_{-i}\leq\tW_{1-i}\leq t^{-1}\tW_{-i}$.  This gives
 $\xi(\tW_{-i})\leq\xi(\tW_{1-i})\leq t\xi(\tW_{-i})$ and so
 $\xi(\tW_{-i})^\perp\geq\xi(\tW_{1-i})^\perp\geq t\xi(\tW_{-i})^\perp$,
 or in other words $t\rfl(\tW)_i\leq\rfl(\tW)_{i-1}\leq\rfl(\tW)_i$.
 This shows that $\rfl(\tW)$ is an unrolled flag, so we have a map
 $\rfl\:X(n)\to X(n)$.

 Next, put $L=\tW_{1-i}\ominus\tW_{-i}$.  As $\xi$ preserves inner
 products we have $\xi(L)=\xi(\tW_{1-i})\ominus\xi(\tW_{-i})$.  Now
 note that for finite-dimensional spaces $P\leq Q$ we have
 $P^\perp\ominus Q^\perp=P^\perp\cap Q^{\perp\perp}=Q\cap
 P^\perp=Q\ominus P$.  After generalising this in accordance with
 Remark~\ref{rem-ominus} we obtain
 $\rfl(\tW)_i\ominus\rfl(\tW)_{i-1}=\xi(L)$, so
 $\bom_i(\rfl(\tW))=\om(\xi(L))$.  However $\om\xi=\xi$ and
 $\om(L)=\bom_{1-i}(\tW)$ so
 $\bom_i(\rfl(\tW))=\bom_{1-i}(\tW)$ as claimed.
\end{proof}

\section{Some combinatorics}
\label{sec-matchings}

First, a basic piece of notation:
\begin{definition}
 We write $N_k$ for $\{1,\dotsc,k\}$.
\end{definition}

Now let $I$ be a totally ordered set with $|I|=2n$ for some $n\in\N$.
(There is of course a unique order-isomorphism $I\simeq N_{2n}$, but
for some later constructions it will be convenient not to make this
identification.)

\begin{definition}\label{defn-sparse}
 We say that a subset $J\sse I$ is \emph{sparse} if for each $j\in J$
 we have $|J^c_{>j}|>|J_{>j}$, where
 \begin{align*}
  J_{>j}   &= \{ k\in J \st k>j \} \\
  J^c_{>j} &= \{ k\in I\sm J\st k>j \}.
 \end{align*}
 (In general, complements are always taken in the set $I$.)
 We write $\SS(I)$, for the set of all sparse subsets, and $\SS_k(I)$ for
 the set of sparse subsets of size $k$.  We also write $\SS(n)$ and
 $\SS_k(n)$ for $\SS(N_{2n})$ and $\SS_k(N_{2n})$.
\end{definition}

\begin{example}
 The sparse subsets of $\{1,2\}$ are $\emptyset$ and $\{1\}$; the sparse subsets of $\{1,2,3,4\}$ are
   $\emptyset,\{1\},\{2\},\{3\},\{1,2\}$ and $\{1,3\}$.
\end{example}

\begin{lemma}\label{lem-sparse-alt}
 A subset $J\sse I$ is sparse iff for every $i\in I$ we have
 $|J^c_{\geq i}|\geq|J_{\geq i}|$.
\end{lemma}
\begin{proof}
 We may assume that $I=N_{2n}$.

 First suppose that $J$ is sparse.  Consider an element $i\in I$.  If
 $J_{\geq i}=\emptyset$ then clearly $|J^c_{\geq i}|\geq|J_{\geq i}|$.
 Suppose instead that the set $J_{\geq i}$ is nonempty, and let $j$ be
 the smallest element.  We then have
 \begin{align*}
  J^c_{\geq i} &= J^c_{>j} \amalg \{i,\dotsc,j-1\}  \\
  J_{\geq i}   &= J_{>j} \amalg \{j\},
 \end{align*}
 so
 \[ |J^c_{\geq i}|-|J_{\geq i}|=(|J^c_{>j}|-|J_{>j}|-1)+(j-i). \]
 The first term is nonnegative by sparsity, and the second is
 nonnegative as $j\in J_{\geq i}$, so the sum is nonnegative as
 required.

 Conversely, suppose we have $|J^c_{\geq i}|\geq|J_{\geq i}|$ for all
 $i\in I$.  For $j\in J$ we have $J^c_{\geq j}=J^c_{>j}$ and
 $J_{\geq j}=J_{>j}\amalg\{j\}$ so the inequality
 $|J^c_{\geq j}|\geq|J_{\geq j}|$ gives $|J^c_{>j}|>|J_{>j}|$ as
 required.
\end{proof}

\begin{lemma}\label{lem-sparse-enumeration}
 Suppose that $J=\{j_1,\dotsc,j_p\}\sse N_{2n}$, with
 $j_1>\dotsb>j_p$.  Then $J$ is sparse iff $j_t<2(n+1-t)$ for all
 $t\in\{1,\dots,p\}$.
\end{lemma}
\begin{proof}
 Clearly $|J_{>j_t}|=t-1$ and
 \[ |J^c_{>j_t}|=|\{j_t+1,\dotsc,2n\}|-|J_{>j_t}| =
      (2n-j_t)-(t-1),
 \]
 so $|J^c_{>j_t}|-|J_{>j_t}|=2(n+1-t)-j_t$.  The claim is clear from
 this.
\end{proof}

\begin{corollary}\label{cor-sparse-small}
 If $J\sse I$ is sparse and $|I|=2n$ then $|J|\leq n$. \qed
\end{corollary}

\begin{proposition}\label{prop-SS-count}
 If $|I|=2n$ and $0\leq p\leq n$ we have
 $|\SS_p(I)|=\bcf{2n}{p}-\bcf{2n}{p-1}$.
\end{proposition}

Any subset $J\sse N_{2n}$ with $|J|=p$ gives a path in $\N^2$ from
$(2n-p,p)$ to $(0,0)$ by the rule
$i\mapsto(|J^c_{\geq i}|,|J_{\geq i}|)$.  Using this kind of
construction we can show that the proposition is equivalent to a
version of the well-known Ballot Problem in combinatorics, which can
be solved using a path reflection principle.  However, this requires
several layers of reindexing and reinterpretation; after unwrapping
those layers, we obtain the more direct proof given below.

\begin{proof}
 The cases $p=0$ and $p=1$ are trivial.  Let $2 \leq p \leq n$.  We
 may again assume $I=N_{2n}$.  Put
 \begin{align*}
  P &= \{ \text{non-sparse subsets } J\subset N_{2n}
           \text{ with } |J|=p \} \\
  Q &= \{ \text{arbitrary subsets } K \subset N_{2n}
           \text{ with } |K|=p-1 \},
 \end{align*}
 so $|P|=\bcf{2n}{p}-|\SS_p(n)|$ and $|Q|=\bcf{2n}{p-1}$.  It will
 suffice to construct a bijection $P\simeq Q$.

 For any set $T\sse I$ and $i\in I$ we put
 $c_T(i)=|T^c_{\geq i}|-|T_{\geq i}|$.  Lemma~\ref{lem-sparse-alt}
 tells us that for $J\in P$ there must exist $i$ with $c_J(i)<0$.  We
 let $a$ denote the smallest such $i$, and then put
 \[ \al(J)=J_{<a}\amalg J^c_{\geq a}\sse I. \]

 Similarly, for $K\in Q$ we note that there exists at least one $i\in I$
 such that $c_K(i)\leq 1$. (For example $2n$ always satisfies this,
 regardless of the choice of $K$.)  We let $b$ denote the smallest
 such $i$, and then put 
 \[ \bt(K)=K_{<b}\amalg K^c_{\geq b}\sse I. \]

 Suppose we have $J\in P$ and we define $a$ as before.  Note that
 $c_J(1)=2(n-p)\geq 0$ so we must have $a>1$.  By the minimality of
 $a$ we have $c_J(a-1)\geq 0$ but $c_J(a)<0$.  Note that 
 \[ c_J(i) - c_J(i+1) = |\{i\}\cap J^c| - |\{i\}\cap J| =
     \begin{cases}
      -1 & \text{ if } i \in J \\
       1 & \text{ if } i \not\in J.
     \end{cases}
 \]
 Using this, we see that $a-1\not\in J$ and $c_J(a-1)=0$ and
 $c_J(a)=-1$.  If we put $q=|J_{\geq a}|$ it follows that
 $1\leq q\leq p$ and 
 \begin{align*}
  |J_{\geq a-1}| &= q     & |J^c_{\geq a-1}| &= q \\
  |J_{\geq a}|   &= q     & |J^c_{\geq a}|   &= q-1 \\
  |J_{< a-1}|    &= p-q   & |J^c_{< a-1}|    &= 2n-p-q \\
  |J_{< a}|      &= p-q   & |J^c_{< a}|      &= 2n+1-p-q.
 \end{align*}
 Now put $K=\al(J)=J_{<a}\amalg J^c_{\geq a}$, so $|K|=(p-q)+(q-1)=p-1$,
 so $K\in Q$.  Note that $K^c_{\geq a}=J_{\geq a}$ and
 $K_{\geq a}=J^c_{\geq a}$ so $c_K(a)=-c_J(a)=1$.  On the other hand,
 for $i<a$ we have
 \begin{align*}
  K_{\geq i}     &= (J_{\geq i} \sm J_{\geq a}) \amalg J^c_{\geq a} \\
  K^c_{\geq i}   &= (J^c_{\geq i} \sm J^c_{\geq a}) \amalg J_{\geq a} \\
  |K_{\geq i}|   &= |J_{\geq i}| - q + (q-1) = |J_{\geq i}| - 1 \\
  |K^c_{\geq i}| &= |J^c_{\geq i}| - (q-1) + q = |J^c_{\geq i}| + 1 \\
  c_K(i) &= c_J(i)+2.
 \end{align*}
 By the definition of $a$ we have $c_J(i)\geq 0$ in this range, so
 $c_K(i)>1$.  This shows that the number $b$ entering in the
 definition of $\bt(K)$ is the same as $a$, and it follows immediately
 that $\bt(K)=J$.

 A very similar analysis shows that for all $K\in Q$ we have
 $\bt(K)\in P$ and $\al\bt(K)=K$, as required.
\end{proof}

\begin{corollary}
 The total number of sparse sets of all sizes is
 \[ \left|\coprod_{p=0}^n SS_p(I)\right| = \bcf{2n}{n}. \qed \]
\end{corollary}

\begin{corollary}
 We have the generating function
 \[ \sum_{n \geq 0} \;\, \sum_{0\leq k\leq n} |\SS_k(n)|s^nt^k =
     \frac{2t}{(t-1)+(t+1)\sqrt{1-4st}}.
 \]
\end{corollary}
\begin{proof}
 Put
 \begin{align*}
  p(s,t) &= \sum_{n \geq 0} \;\, \sum_{0\leq k\leq n} |\SS_k(n)|s^nt^k
          = \sum_{n \geq 0} \;\, \sum_{0\leq k\leq n}
              \left(\bcf{2n}{k}-\bcf{2n}{k-1}\right) s^nt^k \\
  q(s,t) &= \frac{2t}{(t-1)+(t+1)\sqrt{1-4st}} \\
  p^*(s,t) &= p(st,1/t)
          = \sum_{n \geq 0} \;\, \sum_{0\leq k\leq n}
              \left(\bcf{2n}{k}-\bcf{2n}{k-1}\right) s^nt^{n-k} \\
  q^*(s,t) &= q(st,1/t)
            = \frac{2}{(1-t)+(1+t)\sqrt{1-4s}}.
 \end{align*}
 The claim is that $p(s,t)=q(s,t)$.

 It is straightforward to check that $1/q(s,t)$ and $1/q^*(s,t)$ are
 formal power series in $\Q\psb{s,t}$ which reduce to $1$ when $s=0$.
 It follows that all four of the above series lie in $\Q\psb{s,t}$.

 Next, it is formal to check that
 \[ q(s,t) - t\,q(st^2,1/t) = \frac{1-t}{1-s(1+t)^2}. \]
 We claim that $p(s,t)$ has the same property.  Indeed, we have
 \[ t\,p(st^2,1/t) = \sum_{n \geq 0} \;\, \sum_{0\leq k\leq n}
     \left(\bcf{2n}{k}-\bcf{2n}{k-1}\right) s^nt^{2n+1-k}.
 \]
 We can rewrite this in terms of $j=2n+1-k$, using
 $\bcf{2n}{m}=\bcf{2n}{2n-m}$, to get
 \[ t\,p(st^2,1/t) = - \sum_{n \geq 0} \;\, \sum_{n+1\leq j\leq 2n+1}
     \left(\bcf{2n}{j}-\bcf{2n}{j-1}\right) s^nt^j.
 \]
 It follows that
 \[ p(s,t) - t\,p(st^2,1/t) =
    \sum_{n\geq 0}s^n\left(
     \sum_{0\leq k\leq 2n+1}\bcf{2n}{k}t^k -
     \sum_{0\leq k\leq 2n+1}\bcf{2n}{k-1}t^k
    \right).
 \]
 In the first of the inner sums we can drop the $k=2n+1$ term and the
 total is $(1+t)^{2n}$.  In the second of the inner sums we can drop
 the $k=0$ term and shift the indexing to see that the total is
 $t(1+t)^{2n}$.  This just leaves a geometric progression with sum
 $(1-t)/(1-s(1+t)^2)$ as claimed.

 We now put $r(s,t)=p(s,t)-q(s,t)=\sum_{n,i\geq 0}a_{ni}s^nt^i$ say.
 We see that both $r(s,t)$ and $r^*(s,t)=r(st,1/t)$ lie in
 $\Q\psb{s,t}$, which means that $a_{ni}=0$ unless $0\leq i\leq n$.
 We also have $r(s,t)-t\,r(st^2,1/t)=0$, which means that
 $a_{ni}=a_{n,2n+1-i}$.  This clearly gives $r=0$ as required.
\end{proof}

Khovanov in \cite{kh:cmc} introduces the notion of a
\emph{crossingless matching}.  Russell and Tymoczko in \cite{ruty:srk} use
the term \emph{non--crossing matching}.

\begin{definition}\label{defn-matching}
 A \emph{non--crossing matching} on $I$ is a permutation $\tau\:I\to I$
 such that
 \begin{itemize}
  \item[(a)] $\tau^2=1$, and for all $i\in I$ we have $\tau(i)\neq i$.
  \item[(b)] There is no pair $i,j\in I$ with $i<j<\tau(i)<\tau(j)$.
 \end{itemize}
 We write $\NCM(I)$ for the set of non-crossing matchings, or
 $\NCM(n)$ for $\NCM(N_{2n})$.
\end{definition}

\begin{remark}\label{rem-matching-diagram}
 We can draw diagrams for non-crossing matchings as follows.  We draw
 a $2n$-gon with edges labelled by $I$, and we join the midpoint of
 the edge labelled $i$ to the edge labelled $\tau(i)$.  The
 non-crossing condition means that these joining lines do not
 intersect each other.  In the common case where $I=\{1,\dotsc,2n\}$
 we label the vertices with the elements $\{0,\dotsc,2n-1\}$ of
 $\Z/2n$, and the edge joining $i-1$ to $i$ is labelled $i$.  It will
 become clear later why these are natural conventions.  For example,
 the picture
 
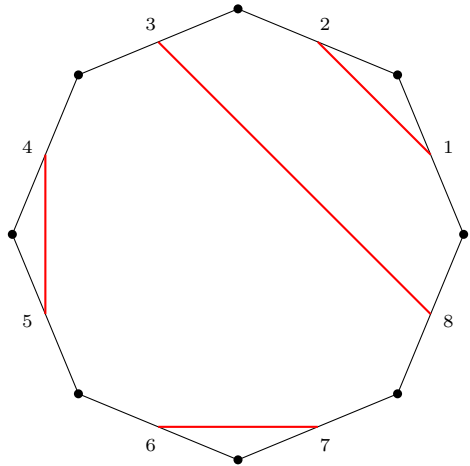
\begin{center}
  \begin{tikzpicture}[scale=3]
   \def\r{0.924}
   \draw (0:1) -- (45:1) -- (90:1) -- (135:1) --
         (180:1) -- (225:1) -- (270:1) -- (315:1) -- (0:1);

   \foreach \i in {0,...,7} {
    \fill (45*\i : 1) circle(0.02);
   }

   \foreach \i in {1,...,8} {
    \draw (-22.5+45*\i :1.01) node{$\ss\i$};
   }

   \draw[thick,red] (22.5:\r) -- ( 67.5:\r);
   \draw[thick,red] (112.5:\r) -- (337.5:\r);
   \draw[thick,red] (157.5:\r) -- (202.5:\r);
   \draw[thick,red] (247.5:\r) -- (292.5:\r);
  \end{tikzpicture}
 \end{center}
 corresponds to the permutation $(1\;2)(3\;8)(4\;5)(6\;7)$.  From this
 perspective it is clear that there is a natural dihedral group action
 on $\NCM(n)$.  However, for some purposes it is better to break the
 symmetry and draw the points $1,\dotsc,2n$ on the $x$-axis, with arcs
 joining $i$ and $\tau(i)$ in the upper half plane.  This results in
 a diagram of the kind introduced by Khovanov in \cite{kh:cmc}*{Figure 1}:
 
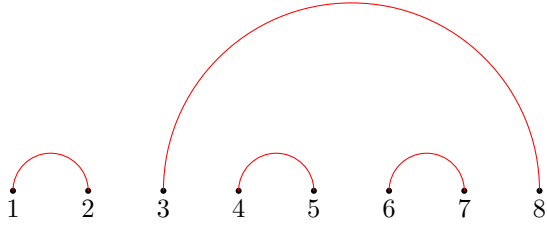
\begin{center}
  \begin{tikzpicture}
   \foreach \i in {1,...,8} {
    \fill (\i,0) circle(0.04);
    \draw (\i,0) node[anchor=north] {$\i$};
   }
   \draw[red] (1.5,0) (2,0) arc(0:180:0.5);
   \draw[red] (5.5,0) (8,0) arc(0:180:2.5);
   \draw[red] (4.5,0) (5,0) arc(0:180:0.5);
   \draw[red] (6.5,0) (7,0) arc(0:180:0.5);
   \draw[white] (0,3) -- (1,3);
  \end{tikzpicture}
 \end{center}
\end{remark}

The following observation will be useful in many places.
\begin{lemma}\label{lem-NCM-parity}
 Let $\tau$ be a non-crossing matching.  Then $\tau(i)-i$ is odd for
 all $i$.
\end{lemma}
\begin{proof}
 After replacing $i$ by $\tau(i)$ if necessary, we may assume that
 $i<\tau(i)$.  The non-crossing condition means that $\tau$ must
 preserve the interval $A=\{i+1,i+2,\dotsc,\tau(i)-1\}$.  As $\tau$ is
 an involution without fixed points, it follows that $|A|$ is even.
 However, we have $|A|=\tau(i)-i-1$, so $\tau(i)-i$ is odd.
\end{proof}

\begin{proposition}\label{prop-NCM-SS}
 There is a bijection $\lm\:\NCM(I)\to\SS_n(I)$ given by
 $\lm(\tau)=\{i\in I\st\tau(i)>i\}$.
\end{proposition}
\begin{remark}
 It follows that
 \[ |\NCM(n)| = |\SS_n(n)| = \bcf{2n}{n} - \bcf{2n}{n-1} =
     \frac{1}{n+1} \bcf{2n}{n},
 \]
 which is the $n$'th Catalan number $C_n$.  Of course there are an
 enormous number of combinatorially defined sets that are known
 to have size $C_n$, many of them listed in the exercises to Chapter~6
 of Stanley's ``Enumerative Combinatorics II''~\cite{st:ecii}.  Some
 of these are visibly in bijection with $\NCM(n)$, for example the set
 of properly-nested bracketings of a word of length $n+1$.
\end{remark}
\begin{proof}
 First suppose we have $\tau\in\NCM(I)$, and we put
 $J=\lm(\tau)=\{i\in I\st \tau(i)>i\}$.  Condition~(a) in
 Definition~\ref{defn-matching} tells us that $I$ divides into $n$
 orbits, each of size two, under the action of $\tau$.  The set $J$
 contains the smaller element of each orbit.  This shows that
 $|J|=n$.  Next, suppose we have $j\in J$.  It is easy to see that for
 $k\in J_{\geq j}$ we have $\tau(k)\in J^c_{>j}$, so $\tau$ gives an
 injective map $J_{\geq j}\to J^c_{>j}$, so
 $|J^c_{>j}|\geq|J_{\geq j}|>|J_{>j}|$.  This proves that
 $J\in\SS_n(I)$, so we at least have a well-defined map
 $\lm\:\NCM(I)\to\SS_n(I)$.

 In the opposite direction, suppose we start with a set
 $J\in\SS_n(I)$.  We define a map $\tau\:J\to J^c$ by decreasing
 recursion using the formula
 \[ \tau(j) = \min(J^c_{>j}\sm\tau(J_{>j})). \]
 More explicitly, if $J=\{j_1<j_2<\dotsb<j_n\}$, we define
 $k_n=\tau(j_n)$ to be the smallest element in $J^c_{>j_n}$, then we
 define $k_{n-1}=\tau(j_{n-1})$ to be the smallest element in
 $J^c_{>j_{n-1}}\sm\{k_n\}$, and so on.  For this to be valid we need
 to know that $J^c_{>j}\sm\tau(J_{>j})$ is nonempty, but that holds
 because
 \[ |J^c_{>j}\sm\tau(J_{>j})| \geq |J^c_{>j}|-|\tau(J_{>j})| =
     |J^c_{>j}|-|J_{>j}| > 0.
 \]
 By construction we have $\tau(j)\not\in\tau(J_{>j})$, which means
 that $\tau\:J\to J^c$ is injective.  As $|J|=|J^c|=n$, we see that
 $\tau$ is actually a bijection from $J$ to $J^c$.  We can thus extend
 $\tau$ over all of $I$ by putting $\tau(\tau(j))=j$, and this gives
 an involution without fixed points.

 Suppose we have elements $i,j$ in $I$ with $i<j<\tau(i)<\tau(j)$.  By
 construction we have $\tau(p)>p$ for all $p\in J$, and so $\tau(p)<p$
 for all $p\in J^c$.  It follows that $i,j\in J$ and
 $\tau(i),\tau(j)\in J^c$.  Next, by construction we have
 $\tau(i)\in\tau(J_{>i})^c\sse\tau(J_{>j})^c$, and by assumption we
 have $\tau(i)>j$, so $\tau(i)\in J^c_{>j}\sm\tau(J_{>j})$, so
 \[ \tau(i) \geq \min(J^c_{>j}\sm\tau(J_{>j})) = \tau(j), \]
 contrary to assumption.  It follows that no such $i$ and $j$ can
 exist, so $\tau\in\NCM(I)$.  We can thus define $\mu\:\SS_n(I)\to\NCM(I)$
 by $\mu(J)=\tau$.  It is clear from the above remarks that
 $\lm(\mu(J))=J$.

 Suppose again that we start with an element $\tau\in\NCM(I)$, and put
 $J=\lm(\tau)$.  Consider an element $k\in J$, and put
 $A=J^c_{>k}\sm\tau(J_{>k})$.  It is clear from the definition of $J$
 that $\tau(k)>k$ and $\tau(k)\not\in J$, so $\tau(k)\in J^c_{>k}$.
 Moreover, as $\tau$ is injective we have
 $\tau(k)\not\in\tau(J_{>k})$, so $\tau(k)\in A$.  Consider another
 element $p\in A$ with $p\neq\tau(k)$.  As $p\in J^c_{>k}$ we have
 $p>k$ and $p\in J^c$, so $p=\tau(j)$ for some $j\in J$.  As
 $p\not\in\tau(J_{>k})$ (because $p\in A$) and $p\neq\tau(k)$ (by
 assumption) we must have $j<k$.  We now have $j<k<\tau(k)$ and
 $\tau(j)=p>k$, so by axiom~(b) in Definition~\ref{defn-matching}, we
 must have $p\geq\tau(k)$.  This proves that
 $\tau(k)=\min(A)=\min(J^c_{>k}\sm\tau(J_{>k}))$.  As $k$ was
 arbitrary, we have $\tau=\mu(\lm(\tau))$ as required.
\end{proof}

\section{A basis for $R(n)$}

In this section we will prove the following result:
\begin{theorem}\label{thm-BR-basis}
 The set
 \[ BR(n) = \{x_J \st J\sse N_{2n} \text{ sparse } \} \]
 is a basis for $R(n)$ over $\Z$.
\end{theorem}
\begin{proof}
 Combine Propositions~\ref{prop-BR-spans} and~\ref{prop-BR-basis}.
\end{proof}

\begin{definition}
 For any finite totally ordered set $I$, we put
 \[ E(I) = \Z[x_i\st i\in I]/(x_i^2\st i\in I). \]
 We regard this as a graded ring with $|x_i|=2$.  The set
 \[ BE(I) = \{x_J\st J\sse I\} \]
 is evidently a basis for $E(I)$ over $\Z$.
\end{definition}

\begin{definition}
 We write $\sg_k(I)$ for the $k$'th elementary symmetric function in
 the variables $x_i$, and we put
 \[ r_I(t) = \prod_{i\in I} (1+tx_i) = \sum_{k=0}^{|I|} \sg_k(I)t^k.
 \]
 We also put
 \[ R(I) = E(I)/(\sg_1(I),\dotsc,\sg_{|I|}(I)). \]
 (so the ring $R(n)$ in the introduction is $R(N_{2n})$.)  We write
 \[ BR(I) = \{x_J\st J\sse I \text{ is sparse }\}. \]
\end{definition}

\begin{lemma}\label{lem-sigma-zero}
 If $J\sse I$ and $k+|J|>|I|$ then $\sg_k(J)=0$ in $R(I)$.
\end{lemma}
\begin{proof}
 It is clear that $r_J(t)r_{J^c}(t)=r_I(t)=1$ in $R(I)[t]$.  Note also
 that $(1+tx_i)(1-tx_i)=1-t^2x_i^2=1$, which implies that
 $r_{J^c}(t)r_{J^c}(-t)=1$.  We can thus multiply both sides of our first
 relation by $r_{J^c}(-t)$ to get $r_J(t)=r_{J^c}(-t)$.  If
 $k>|I|-|J|=|J^c|$ then the coefficient of $t^k$ on the right hand
 side is certainly zero, so the same is true on the left hand side,
 which means that $\sg_k(J)=0$.
\end{proof}

\begin{proposition}\label{prop-BR-spans}
 $R(I)$ is generated as an abelian group by $BR(I)$.
\end{proposition}

The proof will use an ordering of the basis $BE(I)$, which we now
describe.
\begin{definition}\label{defn-order}
 Again let $I$ be a finite, totally ordered set.  We will use the
 lexicographic ordering on subsets of $I$.  In more detail, suppose we
 have subsets $J,K\sse I$ with $J\neq K$.  We list the elements in
 order and thus regard $J$ and $K$ as increasing sequences of elements
 of $I$.  Let $U$ be the longest possible initial segment of $J$ that
 is also an initial segment of $K$.  ($U$ may of course be empty.)
 This means that $J$ is $U$
 followed by some (possibly empty) sequence $J_1$, and $K$ is $U$
 followed by some (possibly empty) sequence $K_1$, where either
 \begin{itemize}
  \item[(a)] $J_1$ is empty and $K_1$ is not; or
  \item[(b)] $K_1$ is empty and $J_1$ is not; or
  \item[(c)] the first entry in $J_1$ is smaller than the first entry
   in $K_1$; or
  \item[(d)] the first entry in $K_1$ is smaller than the first entry
   in $J_1$.
 \end{itemize}
 In cases~(a) and~(c) we declare that $J<K$, and in cases~(b) and~(d)
 we declare that $J>K$.  We transfer this ordering to $BE(I)$ by
 declaring that $x_J<x_K$ iff $J<K$.
\end{definition}

\begin{remark}\label{rem-order}
 In most cases we will only compare sets of the same size.  In this
 situation cases~(a) and~(b) cannot occur, and the rule can be
 restated as follows: we have $J<K$ iff the set
 $D=(J\sm K)\cup(K\sm J)$ is nonempty, and the smallest element of $D$
 lies in $J$.
\end{remark}

Proposition~\ref{prop-BR-spans} follows easily by induction from the
following sharper statement.
\begin{lemma}\label{lem-BR-spans}
 If $J\sse I$ is not sparse then $x_J$ can be expressed in $R(I)$ as a
 $\Z$-linear combination of monomials that are smaller with respect to
 our ordering on $BE(I)$.
\end{lemma}
\begin{proof}
 We will work everywhere with homogeneous terms so only the purely
 lexicographic part of the ordering will be relevant.  Let $j$ be the
 largest index in $J$ for which the sparsity condition is violated.
 Put $K=J_{<j}$ and $L=J_{\geq j}$, so $L$ is still not sparse.  Put
 $p=|L_{>j}|$ and $q=|L^c_{>j}|$ so we must have $q\leq p$.  We claim
 that in fact $p=q$.  This is clear if $p=0$, so suppose that $p>0$,
 so we can let $k$ be the smallest element of $L_{>j}$.  By
 assumption, the sparsity condition is satisfied for $k$, so
 $|L_{>k}|<|L^c_{>k}|$.  Here $|L_{>k}|=p-1$ and $|L^c_{>k}|\leq q$,
 and this can only be consistent if $p=q=|L^c_{>k}|$.  Note that this
 gives $|L|=p+1$ and $|I_{\geq j}|=p+q+1=2p+1$.

 Now put $m=|I_{<j}|$, so $|I|=|I_{<j}|+|I_{\geq j}|=m+1+2p$.  Put
 $M=L\amalg I_{<j}$, and observe that $p+1+|M|=2p+2+m=|I|+1$, so
 $\sg_{p+1}(M)$ maps to zero in $R(I)$ by Lemma~\ref{lem-sigma-zero}.
 If we work in $E(I)$, then the highest term in $\sg_{p+1}(M)$ is
 $x_L$, so the highest term in $x_K\sg_{p+1}(M)$ is $x_Kx_L=x_J$.
 This gives the required relation in $R(I)$.
\end{proof}

\begin{proposition}\label{prop-rho-K}
 Let $K\subset I$ be a sparse set of size $n=|I|/2$, and let $\tau$ be
 the corresponding non-crossing matching.  Then there is a canonical
 isomorphism
 \[ \rho_K\: R(I)/(x_k+x_{\tau(k)}\st k\in K) \xra{\simeq}
     E(K)
 \]
 with $\rho_K(x_k)=-\rho_K(x_{\tau(k)})=x_k$ for all $k\in K$.
\end{proposition}
\begin{proof}
 As $I=K\amalg\tau(K)$, the ring $E(K)$ on the right can be identified
 with the ring
 \[ T = E(I)/(x_k+x_{\tau(k)}\st k\in K). \]
 The left hand side is the quotient of $T$ by the ideal generated by
 the coefficients of $r_I(t)-1$.  It will therefore suffice to show
 that $r_I(t)$ is already equal to 1 in $T[t]$.  It is clear that
 $I$ is the disjoint union of the sets $\{k,\tau(k)\}$ for $k\in K$,
 so $r_I(t)$ is the product of the terms
 $u_k=(1+tx_k)(1+tx_{\tau(k)})$.  As $x_k^2=0$ and $x_{\tau(k)}=-x_k$
 in $T$, we find that $u_k=1$ and so $r_I(t)=1$ as claimed.
\end{proof}

\begin{definition}\label{defn-QI}
 Put $Q(I)=\prod_{K\in\SS_n(I)}E(K)$, and let $\ep_K\in Q(I)$ be the
 idempotent that is $1$ in the $E(K)$ factor and $0$ in all other
 factors, so the set
 \[ BQ(I) = \{x_J\ep_K\st J\sse K\sse I,\; K\in\SS_n(I)\} \]
 is a basis for $Q(I)$ over $\Z$.   The maps $\rho_K$ combine
 to give a single homomorphism $\rho\:R(I)\to Q(I)$, given by
 \[ \rho(a) = \sum_{K\in\SS_n(I)} \rho_K(a)\ep_K. \]
\end{definition}

\begin{lemma}\label{lem-J-plus}
 Suppose that $|I|=2n$ and that $J\sse I$ is sparse with $|J|=m<n$.
 Then the set $J^+=J\cup\{\min(J^c)\}$ is also sparse.
\end{lemma}
\begin{proof}
 It will be harmless to suppose that $I=N_{2n}=\{1,\dotsc,2n\}$.  Put
 $k=\min(J^c)$, so the interval $N_{k-1}=\{1,\dotsc,k-1\}$ is
 contained in $J$, and $J^+=J\cup\{k\}$.  Consider an element
 $j\in J^+$.  If $j>k$ then the sparsity condition for $j\in J$
 immediately implies the sparsity condition for $j\in J^+$.  Suppose
 instead that $1\leq j\leq k$.  Note that $J^+=N_j\amalg(J^+)_{>j}$ so
 $m+1=|J^+|=j+|(J^+)_{>j}|$, so $|(J^+)_{>j}|=m+1-j$.  On the other
 hand, we have $(J^+)^c_{>j}=J^c\sm\{k\}$, so $|(J^+)^c_{>k}|=2n-m-1$,
 so $|(J^+)^c_{>j}|-|(J^+)_{>j}|=2(n-m-1)+j$.  As $n>m$ and $j>0$ we
 see that this is strictly positive, as required.
\end{proof}

\begin{definition}\label{defn-bar}
 Given a sparse set $J\sse I$ with $|J|=m \leq n=|I|/2$, we write $\ov{J}$
 for the sparse set obtained by applying the operation $K\mapsto K^+$
 to $J$, repeated $n-m$ times.  This is easily seen to be
 lexicographically smallest among the sparse sets of size $n$
 containing $J$.
\end{definition}

\begin{proposition}\label{prop-BR-basis}
 If $|I|$ is even then the set $BR(I)$ is linearly independent over
 $\Z$.  Moreover, the map $\rho\:R(I)\to Q(I)$ is a split monomorphism
 of abelian groups.
\end{proposition}
\begin{proof}
 Let $FR(I)$ denote the subgroup of $E(I)$ freely generated by
 $BR(I)$.  The evident map $FR(I)\to R(I)$ is surjective by
 Proposition~\ref{prop-BR-spans}.  We must prove that it is also
 injective.  To see this, we define a total order on the set $BQ(I)$
 as follows: we order the subsets of $I$ lexicographically as before,
 then we declare that $x_J\ep_K<x_{J'}\ep_{K'}$ iff either $J<J'$, or
 ($J=J'$ and $K>K'$).

 We claim that for each sparse set $J$, the highest term in
 $\rho(x_J)$ is $x_J\ep_{\ov{J}}$.  To see this, consider a sparse set
 $K$ of size $n$, and the corresponding permutation
 $\tau=\lm^{-1}(K)\in\NCM(K)$.  Define $\pi\:I\to K$ by $\pi(i)=i$ for
 $i\in K$, and $\pi(i)=\tau(i)<i$ for $i\in I\sm K=\tau(K)$.  Recall
 that $\rho_K(x_i)=x_i$ for $i\in K$, and $\rho_K(x_i)=-x_{\tau(i)}$
 for $i\in I\sm K$.  It follows that $\rho_K(x_J)=0$ if $\pi|_J$ is
 not injective.  On the other hand, if $\pi|_J$ is injective, then
 $\rho_K(x_J)=\pm x_{\pi(J)}$.  If $J\not\sse K$ then $x_{\pi(J)}$ is
 lexicographically strictly smaller than $x_J$.  If $J\sse K$ then
 $\rho_K(x_J)=x_J$, and $K$ is lexicographically at least as large as
 $\ov{J}$.  The claim about highest terms is clear from this.  Now let
 $P$ be the subgroup of $Q(I)$ generated by all basis elements in
 $BQ(I)$ not of the form $x_J\ep_{\ov{J}}$.  We can define a map
 $FR(I)\oplus P\to Q(I)$ by $(a,b)\mapsto\rho(a)+b$, and it is now
 easy to see that this is an isomorphism.  It follows that
 $FR(I)\to R(I)$ is an isomorphism and $\rho\:R(I)\to Q(I)$ is a split
 monomorphism.
\end{proof}

We next make some comments about the relationship between our approach
and that of Russell and Tymoczko.  The issue that we want to explain
is algebraic and combinatorial rather than topological, so we will
just assume for the moment that the map $\phi\:R(n)\to H^*(X(n))$ is
an isomorphism.  Consider a non-crossing matching $\tau\in\NCM(n)$.  A
\emph{dotting} of $\tau$ is a subset $S\sse\lm(\tau)$.  In terms of the
pictures that we drew previously, the left hand end of each arc is an
element of $\lm(\tau)$, and we draw a dot on the arc if the left hand
end lies in $S$.  Thus, in the following picture we have
$\lm(\tau)=\{1,3,4,6\}$ and $S=\{1,4\}$.
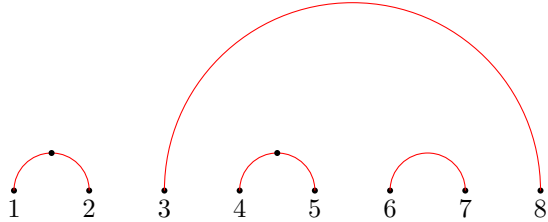
\begin{center}
 \begin{tikzpicture}
  \foreach \i in {1,...,8} {
   \fill (\i,0) circle(0.04);
   \draw (\i,0) node[anchor=north] {$\i$};
  }
  \draw[red] (1.5,0) (2,0) arc(0:180:0.5);
  \draw[red] (5.5,0) (8,0) arc(0:180:2.5);
  \draw[red] (4.5,0) (5,0) arc(0:180:0.5);
  \draw[red] (6.5,0) (7,0) arc(0:180:0.5);
  \fill (1.5,0.5) circle(0.04);
  \fill (4.5,0.5) circle(0.04);
  \draw[white] (0,3) -- (1,3);
 \end{tikzpicture}
\end{center}
We write $\DNCM(n)$ for the set of all dotted non-crossing matchings.
Consider an element $\al=(\tau,S)\in\DNCM(n)$, and put
$K=\lm(\tau)\in\SS_n(n)$.   For any $a\in R(n)$ we let
$\tht_\al(a)$ denote the coefficient of $x_{K\sm S}$ in
$\rho_K(a)$.  This defines an element
$\tht_\al\in\Hom(R(n),\Z)\simeq H_*(X(n))$, of degree $2(n-|K|)$.
\begin{itemize}
 \item[(a)] We say (following Russell and Tymoczko) that $\al$ is
  \emph{standard} if no dotted arc is nested under any undotted arc,
  or equivalently there is no pair $(i,j)$ with $i<j<\tau(j)<\tau(i)$
  and $j\in S$.
 \item[(b)] We say that $\al$ is \emph{costandard} if
  $\al=(\mu(\ov{J}),\ov{J}\sm J)$ for some sparse set $J$, where
  $\ov{J}$ is the lexicographically smallest sparse set of size $n$
  containing $J$, as before.
\end{itemize}
It is clear from our proof of Proposition~\ref{prop-BR-basis} that
$\{\tht_\al\st\al\text{ is costandard }\}$ is a basis for the dual of
$R(n)$.  On the other hand, it follows from the work of Russell and
Tymoczko that $\{\tht_\al\st\al\text{ is standard }\}$ is also a
basis.  We conjecture that $\al$ is standard iff it has the form
$(\mu(J^*),J^*\sm J)$ for some sparse set $J$, where $J^*$ is
lexicographically \emph{largest} among sparse sets of size $n$
containing $J$.  We have checked this by exhaustive computer search
for $n\leq 5$ but we have not attempted to find a proof.

\section{The subvarieties $X(n,K)$}

\begin{definition}
 Let $M$ be a module over $\C[t]/t^n$.  We say that $M$ is
 \emph{balanced} if it is isomorphic to $(\C[t]/t^p)^2$ for some
 $p\leq n$.
\end{definition}

We will investigate this and related concepts in much greater detail
in Section~\ref{sec-torsion}, but the above will do for the moment.

\begin{definition}\label{defn-XK}
 Let $K\sse I$ be a sparse set of size $n$, and let $\tau$ be the
 corresponding non-crossing matching (so $K=\{i\st \tau(i)>i\}$).  We
 say that a flag $\un{W}\in X(n)$ is \emph{$K$-balanced} if for each
 $i\not\in K$ the quotient $W_i/W_{\tau(i)-1}$ is balanced.  We let
 $X(n,K)\sse X(n)$ denote the space of $K$-balanced flags.  One can
 check that this is a closed subvariety of $X(n)$.  We also define
 \[ \bom_K\:X(n,K)\to \prod_{k\in K} \CP^1 \simeq (S^2)^n \]
 by
 \[ \bom_K(\un{W})_k = \bom_k(\un{W}) =
      \om(W_k\ominus W_{k-1})
 \]
 (where $\om$ is as in Definition~\ref{defn-om}).
\end{definition}

\begin{remark}
 In terms of the map $\bom\:X(n)\to(S^2)^{2n}$ defined in
 Proposition~\ref{prop-om-emb}, our map $\bom_K$ is the composite
 \[ X(n,K) \xra{\text{inc}} X(n)
     \xra{\bom} \prod_{i\in N_{2n}}\CP^1
     \xra{\text{proj}} \prod_{k\in K}\CP^1.
 \]
\end{remark}

\begin{proposition}\label{prop-XK}
 The map $\om_K\:X(n,K)\to\prod_{k\in K}\CP^1$ is a homeomorphism.
\end{proposition}

The proof will follow after some preliminaries.
\begin{definition}\label{defn-XKm}
 For $0\leq m\leq 2n$ we let $X(n,K,m)$ be the space of partial flags
 \[ 0 = W_0 < W_1 < \dotsb < W_m \]
 such that
 \begin{itemize}
  \item[(a)] $\dim(W_i)=i$ for all $i$ with $0\leq i\leq m$.
  \item[(b)] $tW_i\leq W_{i-1}$ for all $i$ with $0<i\leq m$.
  \item[(c)] $W_i/W_{\tau(i)-1}$ is balanced for all $i\not\in K$ with
   $0<i\leq m$.
 \end{itemize}
 We let $\pi_m\:X(n,K,m)\to X(n,K,m-1)$ be the obvious projection.
\end{definition}

Note that $X(n,K,0)$ is a point and $X(n,K,2n)=X(n,K)$.  It will thus be
sufficient to prove that $\pi_m$ is an homeomorphism for $m\not\in K$,
and that
\[ (\pi_m,\bom_m)\:X(n,K,m) \to X(n,K,m-1)\tm \CP^1 \]
is an homeomorphism for $m\in K$.  This will be done in
Lemmas~\ref{lem-pi-iso} and~\ref{lem-pi-PU}.

\begin{lemma}\label{lem-count-K}
 Suppose that $1\leq m\leq 2n$ and $m\not\in K$, and put $p=\tau(m)$.
 Then
 \begin{itemize}
  \item[(i)] $p=m-2d+1$ for some $d>0$.
  \item[(ii)] $\tau$ preserves the intervals $\{p,\dotsc,m\}$
   and $\{p+1,\dotsc,m-1\}$.
  \item[(iii)] If $|K_{<m}|=r$ then $|K_{<p}|=r-d$.
  \item[(iv)] If $d>1$ then $m-1\not\in K$.
 \end{itemize}
\end{lemma}
\begin{proof}
 As $m\not\in K$ we must have $p<m$.  The non-crossing condition
 implies that $\tau$ preserves the set $P=\{p+1,\dotsc,m-1\}$.  As
 $\tau$ exchanges $p$ and $m$ we see that it also preserves the set
 $Q=\{p,\dotsc,m\}$.  As $\tau$ is an involution without fixed points
 we see that $|P|$ and $|Q|$ must be even, so $p=m-2d+1$ for some
 $d>0$.  Using the description $K=\{k\st\tau(k)>k\}$ we also see that
 $|Q\cap K|=|Q|/2=d$.  It follows that
 \[ |K_{<p}| = |K_{<m}| - |Q\cap K| = r-d. \]
 Finally, if $d>1$ then $m-1$ is the largest element in the nonempty
 interval $P$, and $\tau$ preserves $P$, so we must have
 $\tau(m-1)<m-1$, so $m-1\not\in K$.
\end{proof}

\begin{lemma}\label{lem-W-exponent}
 Let $\un{W}$ be a point in $X(n,K,m)$ and put $r=|K_{\leq m}|$.  Then
 $t^rW_m=0$.
\end{lemma}
\begin{proof}
 First suppose that $m\in K$, so $|K_{\leq m-1}|=r-1$.  By induction we
 have $t^{r-1}W_{m-1}=0$, and by axiom~(b) we have $tW_m\leq W_{m-1}$,
 so $t^rW_m=0$ as required.

 Suppose instead that $m\not\in K$.  Then we have $\tau(m)=m-2d+1$ for
 some $d>0$, and $W_m/W_{m-2d}$ is balanced, so
 $t^dW_m\leq W_{m-2d}$.  By induction we have $t^sW_{m-2d}=0$, where
 $s=|K_{\leq m-2d}|$, so $t^{d+s}W_m=0$.  Moreover,
 Lemma~\ref{lem-count-K}(iii) gives $s=r-d$, so $t^rW_m=0$ as claimed.
\end{proof}

\begin{lemma}\label{lem-interval}
 Let $\un{W}$ be a point in $X(n,K,m)$.  Suppose that $0\leq p<q\leq m$
 and that $\tau$ preserves the interval $\{p+1,\dotsc,q\}$.  Then
 $q-p$ is even and $W_q/W_p$ is balanced (so $t^{(q-p)/2}W_q\leq W_p$).
\end{lemma}
\begin{proof}
 As $\tau$ preserves the interval and has no fixed points we must have
 $\tau(q)<q$, so $q\not\in K$.  Put $r=\tau(q)-1$ so $p\leq r<q$.  By
 axiom~(c) we see that $W_q/W_r$ is balanced, so $q-r$ is even and
 $t^{(q-r)/2}W_q\leq W_r$.  If $r=p$ then we are done.  Otherwise, the
 non-crossing condition implies that $\tau$ preserves the interval
 $\{p+1,\dotsc,r\}$ so we may assume by induction that $W_r/W_p$ is
 balanced, so $r-p$ is even and $t^{(r-p)/2}W_r\leq W_p$.  It follows
 that $t^{(q-p)/2}W_q\leq W_p$.  As $W_q/W_p$ is a subquotient of
 $V(n)$ it must be isomorphic to $\C[t]/t^a\oplus\C[t]/t^b$ for some
 $a$ and $b$.  As $\dim(W_q/W_p)=q-p$ and $t^{(q-p)/2}(W_q/W_p)=0$ we
 must have $a=b=(q-p)/2$, so $W_q/W_p$ is balanced.  (Some of this is
 explained in more detail in Section~\ref{sec-torsion}.)
\end{proof}

\begin{lemma}\label{lem-pi-iso}
 Suppose that $m\not\in K$, so $\tau(m)=m-2d+1$ for some $d>0$.  Then
 $\pi_m\:X(n,K,m)\to X(n,K,m-1)$ is an isomorphism of varieties.
 Moreover, if $d>1$ then for $(W_0,\dotsc,W_m)\in X(n,K,m)$ we have
 $t^{n-d}W_{m-2d}=0$ and $W_m=t^{-d}W_{m-2d}$ and
 $W_{m-1}=t^{1-d}W_{m-2d+1}$.
\end{lemma}
\begin{proof}
 Consider a point $\un{W}\in X(n,K,m-1)$.  As $m\not\in K$ we have
 $\tau(m)=m-2d+1$ for some $d>0$.  The only possible way to
 construct a preimage in $X(n,K,m)$ is to take
 \[ W_m = t^{-d}W_{m-2d} = \{v\in V(n)\st t^dv\in W_{m-2d}\}. \]
 We must check that this satisfies $tW_m\leq W_{m-1}\leq W_m$ and
 $\dim(W_m)=m$.  Put $r=|K_{\leq m}|\leq |K|=n$.  By
 Lemmas~\ref{lem-count-K} and~\ref{lem-W-exponent} we have $r\geq d$
 and $t^{r-d}W_{m-2d}=t^rW_m=0$, so $t^{n-d}W_{m-2d}=0$, so
 $W_{m-2d}\leq t^dV(n)$.  This means we have a short exact sequence
 \[ t^{n-d}V(n) \xra{} W_m \xra{t^d} W_{m-2d} \]
 showing that $\dim(W_m)=\dim(W_{m-2d}) + 2d = m$.

 For the rest of the proof we separate the cases $d=1$ and $d>1$.
 Suppose first that $d=1$, so $\tau(m)=m-1$ and $W_m=t^{-1}W_{m-2}$.
 As the original sequence $\un{W}$ is in $X(n,K,m-1)$ we have
 $tW_{m-1}\leq W_{m-2}$, which gives $W_{m-1}\leq t^{-1}W_{m-2}=W_m$.
 As $\un{W}\in X(n,K,m-1)$ we also have $W_{m-2}\leq W_{m-1}$, so
 $tW_m\leq W_{m-2}\leq W_{m-1}$.  It follows that the extended
 sequence $\un{W}^+=(W_0,\dotsc,W_m)$ lies in $X(n,K,m)$.

 Suppose instead that $d>1$.  Now $\tau$ preserves the nonempty
 interval $\{m-2d+2,\dotsc,m-1\}$, so by Lemma~\ref{lem-interval} we
 have $W_{m-1}=t^{1-d}W_{m-2d+1}$.  Thus, if $v\in W_{m-1}$ we have
 $t^dv=tt^{d-1}v\in tW_{m-2d+1}\leq W_{m-2d}$, so
 $v\in t^{-d}W_{m-2d}=W_m$.  On the other hand, for $v\in W_m$ we have
 $t^{d-1}tv=t^dv\in W_{m-2d+1}$ so $tv\in t^{1-d}W_{m-2d+1}=W_{m-1}$.
 This proves that $tW_m\leq W_{m-1}<W_m$ as required, so again
 $\un{W}^+\in X(n,K,m)$.

 We now see that $\pi_m$ is a bijection.  The key ingredient in the
 inverse is the map
 $\text{Grass}_{m-2d}(t^dV(n))\to\text{Grass}_m(V(n))$ given by
 $U\mapsto t^{-d}U$.  It is standard that this is a morphism of
 varieties, and it follows that $\pi_m$ is an isomorphism of
 varieties.
\end{proof}

\begin{lemma}\label{lem-pi-PU}
 Suppose that $m\in K$.  Then there is a two-dimensional algebraic
 vector bundle $U$ over $X(n,K,m-1)$ with fibres
 $U_{\un{W}}=(t^{-1}W_{m-1})/W_{m-1}$, and $X(n,K,m)$ is isomorphic (as
 a variety) to the associated projective bundle $PU$.  Moreover, $U$ is
 topologically trivial, so $X(n,K,m)$ is homeomorphic to
 $X(n,K,m-1)\tm\CP^1$.
\end{lemma}
\begin{proof}
 As $m\in K$ we have $|K_{\leq m-1}|<|K|=n$, so
 Lemma~\ref{lem-W-exponent} gives $t^{n-1}W_{m-1}=0$ and so
 $W_{m-1}\leq tV(n)$.  Given this we see that
 $\dim(t^{-1}W_{m-1})=m+1$ and standard techniques show that there is
 an algebraic vector bundle $U$ with fibres as described.  For any
 $\un{W}\in X(n,K,m-1)$ we see that the preimages in $X(n,K,m)$ biject
 with the spaces $W_m$ satisfying $W_{m-1}<W_m<t^{-1}W_{m-1}$, and
 thus with the one-dimensional subspaces of $U_{\un{W}}$.  This gives
 a bijection from $X(n,K,m)$ to $PU$, and again it is a standard piece
 of algebraic geometry to see that this is an isomorphism of
 varieties.  In the topological category we can identify $U_{\un{W}}$
 with $(t^{-1}W_{m-1})\ominus W_{m-1}$, and
 Proposition~\ref{prop-isometry} tells us that the map $\om$ gives an
 isomorphism from this space to $\C^2$.  This gives the required
 trivialisation of $U$, proving that $X(n,K,m)=X(n,K,m-1)\tm\CP^1$.
\end{proof}

\begin{proof}[Proof of Proposition~\ref{prop-XK}]
 Combine Lemmas~\ref{lem-pi-iso} and~\ref{lem-pi-PU}.
\end{proof}

\begin{definition}\label{defn-chi}
 We define $\chi\:\CP^1\to\CP^1$ by $\chi(L)=L^\perp$, so
 $\chi([z:w])=[-\ov{w}:\ov{z}]$.  There is a standard homeomorphism
 $f\:\CP^1\to S^2$ given by
 \[ f([z:w]) =
     (2\text{Re}(z\ov{w}),2\text{Im}(z\ov{w}),|w|^2-|z|^2)/
       (|w|^2+|z|^2).
 \]
 If we use this to identify $\CP^1$ with $S^2$, we find that
 $\chi(u)=-u$.
\end{definition}

\begin{lemma}\label{lem-tau-chi}
 Suppose that $\un{W}\in X(n,K)$.  Then for $1\leq m\leq 2n$ we have
 $\bom_{\tau(m)}(\un{W})=\chi(\bom_m(\un{W}))$.
\end{lemma}
\begin{proof}
 We write $M_i=W_i\ominus W_{i-1}$, so the claim is that
 $\om(M_{\tau(m)})=\chi(\om(M_m))$.  As $\chi^2=1$ and $\tau^2=1$ and
 $\tau$ exchanges $K$ and $K^c$, it will suffice to treat the case
 where $m\not\in K$.  We then have $\tau(m)=m-2d+1$ for some $d>0$.

 First consider the case where $d=1$, so $\tau(m)=m-1$.  We then have
 $W_m=t^{-1}W_{m-2}$, so the spaces $M_{m-1}$ and $M_m$ are orthogonal
 complements to each other in the space $U=W_m\ominus W_{m-2}$.
 Proposition~\ref{prop-isometry} implies that $\om\:U\to\C^2$ is an
 isometric isomorphism, so $\om(M_{m-1})$ and $\om(M_m)$ are
 orthogonal complements, as required.

 Now suppose instead that $d>1$, and put $p=\tau(m)=m-2d+1$.  By the
 second half of Lemma~\ref{lem-pi-iso} we have $t^{n-d}W_{p-1}=0$ and
 $W_m=t^{-d}W_{p-1}$ and $W_{m-1}=t^{1-d}W_p$.  Put
 $T=(t^{-1}W_{p-1})\ominus W_{p-1}$, so $M_p\leq T$.  Note that $M_m$
 is orthogonal to $W_{m-1}=t^{1-d}W_p$ and therefore also to
 $t^{1-d}\{0\}=t^{n-d+1}V(n)$.  It follows that any generator
 $a\in M_m$ can be written as $a=\sum_{i=0}^{n-d}a_it^i$ with
 $a_i\in\C^2$.  From this it follows that $t^{d-1}M_m\neq 0$ and also
 that $\om(t^{d-1}M_m)=\om(M_m)$.
 
 We claim that $t^{d-1}M_m$ is orthogonal to $W_p$.  To see this, consider
 an element $b\in W_p$ and an element $a \in t^{d-1}M_m$.
 As $t^{n-d}W_{p-1}=0$ we have $t^{n-d+1}b=0$, so $b$ can be written
 as $b=\sum_{i=d-1}^nb_it^i$ for some $b_i\in\C^2$.  Put
 $b'=\sum_{i=d-1}^nb_it^{i-d+1}$, so that $t^{d-1}b'=b$ and so
 $b'\in t^{1-d}W_p=W_{m-1}$.  Now $M_m$ is orthogonal to $W_{m-1}$ so
 $\ip{a,b'}=0$, but it is clear from the defining formulae that
 $\ip{t^{d-1}a,b}=\ip{a,b'}$ so $\ip{t^{d-1}a,b}=0$ as required.
 
 On the other hand, we have $t^dM_m\leq t^dW_m=W_{p-1}$, so
 $t^{d-1}M_m\leq t^{-1}W_{p-1}$.  As $t^{d-1}M_m$ is contained in
 $t^{-1}W_{p-1}$ and orthogonal to $W_p=W_{p-1}\oplus M_p$ we see that
 $t^{d-1}M_m$ is contained in $T$ and in fact is equal to
 $T\ominus M_p$.  As $\om$ gives an isometric isomorphism $T\to\C^2$
 we deduce that $\om(M_m)=\om(t^{d-1}M_m)=\chi(\om(M_p))$ as required.
\end{proof}

\begin{corollary}\label{cor-XK}
 The maps $\bom_i$ together give a homeomorphism from $X(n,K)$ to
 the space
 \[ X'(n,K) = \{u\:\{1,\dotsc,2n\}\to\CP^1\st
               u(\tau(i))=\chi(u(i)) \text{ for all } i\}.
 \]
\end{corollary}
\begin{proof}
 Combine Proposition~\ref{prop-XK} and Lemma~\ref{lem-tau-chi}.
\end{proof}

\begin{corollary}
 We have a commutative diagram as follows:
 \[ \xymatrix{
  R(n) \ar[r]^{\rho_K} \ar@{ >->}[d]_{\phi} &
  E(K) \ar[d]^{\simeq} \\
  H^*(X(n)) \ar[r]_{\res} &
  H^*(X(n,K))
 } \]
 Moreover, the map $\phi\:R(n)\to H^*(X(n))$ is a split monomorphism
 of abelian groups.
\end{corollary}
\begin{proof}
 The map $\chi\:\C\cup\{\infty\}\to\C\cup\{\infty\}$ is the composite of
 the holomorphic isomorphism $z\mapsto -1/z$ (which has degree one)
 with the map $z\mapsto\ov{z}$ (which has degree $-1$), so
 $\text{deg}(\chi)=-1$.  Given this, we deduce from
 Lemma~\ref{lem-tau-chi} that $\res(x_i+x_{\tau(i)})=0$ in $H^2(X(n,K))$
 for all $i$.  In combination with Proposition~\ref{prop-rho-K},
 this implies that there is a unique map $E(K)\to H^*(X(n,K))$ making
 the diagram commute.  It is clear from Proposition~\ref{prop-XK} that
 this map is an isomorphism.  We know from
 Proposition~\ref{prop-BR-basis} that the composite
 \[ R(n) \xra{\phi} H^*(X(n)) \xra{\res} \prod_K H^*(X(n,K)) \]
 is a split monomorphism.  We can compose any splitting with the map
 $\res$ to get a splitting of $\phi$.
\end{proof}

\section{The subvarieties $X(n,i)$}
\label{sec-pinch}

\begin{definition}
 For any $i$ with $0<i<2n$, we put
 \[ X(n,i) =
     \{\un{W}\in X(n) \st W_{i+1}/W_{i-1}\text{ is balanced} \} =
     \{\un{W}\in X(n) \st tW_{i+1}=W_{i-1}\}.
 \]
\end{definition}

We will show that $X(n,i)$ is homeomorphic to $X(n-1)\tm\CP^1$.  For
this we let $\dl$ be the evident isomorphism $tV(n)\to V(n-1)$ given
by $\dl(t^{i+1}a)=t^ia$ for all $i<n-1$ and $a\in\C^2$.  We also let
$\pi\:V(n)\to V(n-1)$ be the evident projection (so $\pi(v)=\dl(tv)$).

\begin{lemma}\label{lem-Xni}
 If $\un{W}\in X(n,i)$ then
 \begin{itemize}
  \item[(a)] $W_j\leq tV(n)=\dom(\dl)$ for all $j<i$
  \item[(b)] $W_j\geq t^{n-1}V(n)=\ker(\pi)$ for all $j>i$
  \item[(c)] $W_{i-1}=tW_{i+1}$
  \item[(d)] $W_{i+1}=t^{-1}W_{i-1}$.
 \end{itemize}
\end{lemma}
\begin{proof}
 By assumption, the quotient $W_{i+1}/W_{i-1}$ is balanced, so
 $W_{i+1}\leq t^{-1}W_{i-1}$.  From the short exact sequence
 $t^{n-1}V(n)\to V(n)\xra{t}tV(n)$ we see that
 $\dim(t^{-1}W_{i-1})=2+\dim(W_{i-1}\cap tV(n))$.  However, as
 $\un{W}\in X(n)$ we have $\dim(W_{i+1})=i+1$.  This can only be
 consistent if $W_{i-1}\leq tV(n)$ and $W_{i+1}=t^{-1}W_{i-1}$, in
 which case $W_{i+1}\geq t^{-1}\{0\}=t^{n-1}V(n)$.  The inclusion
 $W_{i-1}\leq tV(n)$ also implies that
 $W_{i-1}=tt^{-1}W_{i-1}=tW_{i+1}$.  We have now proved~(c), (d), the
 case $j=i-1$ of~(a) and the case $j=i+1$ of~(b).  The remaining cases
 of~(a) and~(b) follow because $W_j\leq W_{i-1}$ when $j<i$, and
 $W_j\geq W_{i+1}$ when $j>i$.
\end{proof}

\begin{proposition}\label{prop-Xni-bundle}
 There is a morphism of varieties $\lm\:X(n,i)\to X(n-1)$ given by
 \[ \lm(W_0,\dotsc,W_{2n}) =
     \left(\dl(W_0),\dotsc,\dl(W_{i-1})=\pi(W_{i+1}),
           \pi(W_{i+2}),\dotsc,\pi(W_{2n})\right).
 \]
 This lifts to give an isomorphism $X(n,i)\to PU$, where $U$ is the
 vector bundle over $X(n-1)$ whose fibre at $\un{W}'$
 is $\pi^{-1}(W'_{i-1})/\dl^{-1}(W'_{i-1})$.  Moreover, this bundle is
 topologically trivial, so $X(n,i)$ is homeomorphic to
 $X(n-1)\tm\CP^1$.
\end{proposition}
\begin{proof}
 Part~(a) of Lemma~\ref{lem-Xni} shows that $\dl(W_j)$ is defined for
 $j<i$, and $\dl\:tV(n)\to V(n-1)$ is an isomorphism so
 $\dim(\dl(W_j))=j$ in this range.  Part~(c) together with the
 relation $\pi(v)=\dl(tv)$ shows that $\dl(W_{i-1})=\pi(W_{i+1})$, as
 indicated above.  Part~(b) implies that for $j>i$ we have
 $\dim(\pi(W_j))=j-2$.  It is now clear that we have a morphism
 $\lm\:X(n,i)\to X(n-1)$ as described.  Now write $\un{W}'$ for
 $\lm(\un{W})$.  For $j<i$ we find that $W_j=\dl^{-1}(W'_j)$, and for
 $j>i$ we have $W_j=\pi^{-1}W'_{j-2}$, so all these spaces are determined
 by $\lm(\un{W})$.  However, $W_i$ can be any space with
 $\dl^{-1}(W'_{i-1})<W_i<\pi^{-1}(W'_{i-1})$, and such subspaces
 biject with one-dimensional subspaces of $U_{\un{W}'}$.  This shows
 that $\lm$ lifts to a bijection $X(n,i)\to PU$; we leave it to the
 reader to check that this is actually an isomorphism of varieties.
 Moreover, it follows from Proposition~\ref{prop-isometry} that
 $\om\:U_{\un{W}'}\to\C^2$ is an isomorphism for all $\un{W}'$, which
 gives the required trivialisation.
\end{proof}

\begin{lemma}\label{lem-Xni-cover}
 $X(n)=\bigcup_{i=1}^{n}X(n,i)$.
\end{lemma}
For our applications later it would be sufficient to know that
$X(n)$ is the union of all the subvarieties $X(n,i)$ for
$1\leq i\leq 2n$, but in fact we need only the first $n$ of these, and
the proof is no harder.
\begin{proof}
 Consider a point $\un{W}\in X(n)$.  It is clear that $W_1$ is a
 cyclic $\C[t]$-module, but $W_{n+1}$ cannot be.  Let $i$ be the
 largest index such that $W_i$ is cyclic, so $1\leq i\leq n$.  Choose
 a generator $a\in W_i$, and note that $ta$ must generate $W_{i-1}$.
 Then choose $b\in W_{i+1}\sm W_i$.  Note that $tb\in W_i$, but $tb$
 cannot generate $W_i$ otherwise $W_{i+1}$ would be cyclic.  This
 means that $tb$ must lie in $tW_i$, so after adjusting our choice of
 $b$ we can assume that $tb=0$.  It now follows that $a$ and $b$ give
 a basis for $W_{i+1}/W_{i-1}$, and both of them are annihilated by $t$, so
 $W_{i+1}/W_{i-1}$ is balanced, so $\un{W}\in X(n,i)$.
\end{proof}

\begin{proposition}\label{prop-XK-union}
 $X(n)$ is the union of the subvarieties $X(n,K)$ (as $K$ runs over
 all sparse subsets of size $n$).
\end{proposition}
\begin{proof}
 Consider a point $\un{W}\in X(n)$.  Choose $i$ such that
 $\un{W}\in X(n,i)$, and put $\un{W}'=\lm(\un{W})\in X(n-1)$.  By
 induction on $n$, we may assume that $\un{W}'\in X(n-1,K')$ for some
 sparse set $K'$, corresponding to a non-crossing matching $\tau'$ on
 $N_{2n-2}$.  Let $\sg$ be the order-preserving bijection
 $N_{2n}\sm\{i,i+1\}\to N_{2n-2}$ and define $\tau\:N_{2n}\to N_{2n}$
 by
 \[ \tau(j) =
     \begin{cases}
      \sg^{-1}\tau'\sg(j) & \text{ if } j\not\in \{i,i+1\} \\
      i+1 & \text{ if } j=i \\
      i   & \text{ if } j=i+1.
     \end{cases}
 \]
 One can check that this is a non-crossing matching, corresponding to
 the sparse set $K=\sg^{-1}(K')\amalg\{i\}$.  We claim that
 $\un{W}\in X(n,K)$.  To see this, consider a point
 $j\in N_{2n}\sm K$, and put $k=\tau(j)\in K$ (so $k<j$).  If both $j$
 and $k$ are less than $i$ then $j\in K'$ and $\dl$ induces an
 isomorphism $W_j/W_{k-1}\to W'_j/W'_{k-1}$ so $W_j/W_{k-1}$ is
 balanced.  If $j=i+1$ then $k=i$ so $W_j/W_{k-1}=W_{i+1}/W_{i-1}$,
 which is balanced because $\un{W}\in X(n,i)$.  If both $j$ and $k$
 are larger than $i+1$ then $k-2\in K'$ with $\tau'(k-2)=j-2$, and
 $\pi$ induces an isomorphism $W_j/W_{k-1}\to W'_{j-2}/W'_{k-3}$, so
 $W_j/W_{k-1}$ is again balanced.  This just leaves the case where
 $j>i+1$ but $k<i$.  Here $k\in K'$ with $\tau'(k)=j-2$ so
 $W'_{j-2}/W'_{k-1}$ is balanced.  Moreover, we have
 $W'_{k-1}=\dl(W_{k-1})$ and $W'_{j-2}=\pi(W_j)=\dl(tW_j)$, so $\dl$
 induces an isomorphism $tW_j/W_{k-1}\to W'_{j-2}/W'_{k-1}$, so
 $tW_j/W_{k-1}$ is balanced.  We also know from Lemma~\ref{lem-Xni}
 that $W_j$ contains $t^{n-1}V(n)=t^{-1}\{0\}$, and it follows from
 this that $W_j/tW_j$ is balanced.  As $W_j/tW_j$ and $tW_j/W_{k-1}$
 are balanced we see that $W_j/W_{k-1}$ is also balanced.  This
 completes the proof that $\un{W}\in X(n,K)$, as required.
\end{proof}

\begin{corollary}\label{prop-X-prime}
 Put $X'(n)=\bigcup_KX'(n,K)$, where $K$ runs over the sparse sets of
 size $n$ in $N_{2n}$, and $X'(n,K)$ is as in Corollary~\ref{cor-XK}.
 Then $\bom$ gives a homeomorphism $X(n)\to X'(n)$.
\end{corollary}
\begin{proof}
 Combine Proposition~\ref{prop-om-emb}, Corollary~\ref{cor-XK} and
 Proposition~\ref{prop-XK-union}.
\end{proof}

\section{Torsion modules}
\label{sec-torsion}

In the previous sections we defined subvarieties $X(n,K)$ and $X(n,i)$
of $X(n)$.  By taking various intersections of these we can define
many more subvarieties.  To understand all of these we need some
algebraic theory of torsion modules, which will be developed in the
present section, and some graph theory.

\begin{definition}\label{defn-torsion}
 Put $A=\C[t]$.  By a \emph{torsion module} we mean an $A$-module $M$
 such that $\dim_\C(M)<\infty$ and $t^nM=0$ for some $n$.  For any
 such module we put $M[k]=\{m\in M\st t^km=0\}$.  It is well-known
 that any torsion module is isomorphic to a direct sum of copies of
 the modules $A_p=A/t^p$.  The number of summands is
 \[ \rho(M) = \dim_\C(M/tM) = \dim_\C(M[1]). \]
 We call this number the \emph{rank} of $M$.
\end{definition}

\begin{remark}\label{rem-rank}
 Suppose we have torsion modules $K\leq L$.  As $K[1]\leq L[1]$ we can
 use the description $\rho(M)=\dim(M[1])$ to see that
 $\rho(K)\leq\rho(L)$.  Similarly, we can use the description
 $\rho(M)=\dim(M/tM)$ to see that $\rho(L/K)\leq\rho(L)$.
\end{remark}

\begin{definition}\label{defn-thin}
 We say that a torsion module $M$ is \emph{thin} if $\rho(M)\leq 2$.
\end{definition}

Remark~\ref{rem-rank} shows that submodules and quotients of thin
modules are thin.  Thus, for any point $\uW\in X(n)$ and any
$i\leq j\leq 2n$ we see that $W_i$ and $W_j/W_i$ are thin.  From the
general classification of torsion modules, we see that any thin module
is isomorphic to $A_p\oplus A_{p+q}$ for a unique pair $p,q\geq 0$.

\begin{definition}\label{defn-imbalance}
 Let $M\simeq A_p\oplus A_{p+q}$ be a thin module.  We put
 \begin{align*}
  \eta(M) &= \min\{k\st t^kM=0\} = p+q \\
  \bt(M) &= \max\{k\st\dim(M[k])=2k\}
          = \max\{k\st \dim(M/t^kM)=2k\} = p \\
  \dl(M) &= 2\eta(M)-\dim(M) = q.
 \end{align*}
 We call $\eta(M)$ and $\dl(M)$ the \emph{exponent} and
 \emph{imbalance} of $M$.  We say that $M$ is \emph{balanced} if
 $\dl(M)=0$.
\end{definition}

\begin{remark}\label{rem-beta}
 For any thin module $M$ it is clear that
 $\dim(M[k]),\dim(M/t^kM)\leq 2k$.  Thus, the definition of $\bt$
 could be rewritten as
 \[ \bt(M) = \max\{k\st\dim(M[k])\geq 2k\}
           = \max\{k\st \dim(M/t^kM)\geq 2k\}.
 \]
\end{remark}

\begin{lemma}\label{lem-beta}
 Let $K\leq L$ be thin modules.  Then $\bt(K),\bt(L/K)\leq\bt(L)$.
\end{lemma}
\begin{proof}
 Put $p=\bt(K)$, so $\dim(K[p])=2p$.  As $K[p]\leq L[p]$ we have
 $\dim(L[p])\geq 2p$, so $\bt(L)\geq p$.  Now put $M=L/K$ and
 $q=\bt(M)$, so $\dim(M/t^qM)=2q$.  Here $M/t^qM$ is a quotient of
 $L/t^qL$, so $\dim(L/t^qL)\geq 2q$, so $\bt(L)\geq q$.
\end{proof}

\begin{lemma}\label{lem-gap}
 Suppose we have thin modules $K\leq L\leq M$ such that
 $\dim(L/K)=2d$.  Then the following are equivalent:
 \begin{itemize}
  \item[(a)] $L/K$ is balanced.
  \item[(b)] $\bt(L)\geq d$ and $K=t^dL$.
  \item[(c)] $\bt(M/K)\geq d$ and $L=t^{-d}K$.
 \end{itemize}
 Moreover, if these conditions hold then we have inclusions
 \[ \xymatrix{
  0 \ar@{ >->}[d] \ar@{ >->}[r] &
  K \ar@{ >->}[d] \ar@{ >->}[r] &
  t^dM  \ar@{ >->}[d] \\
  M[d]  \ar@{ >->}[r] &
  L  \ar@{ >->}[r] &
  M
 } \]
 and the following statements also hold:
 \begin{align*}
  \bt(M) & \geq d \\
  \dl(M) &= \dl(t^dM) = \dl(M/M[d]) \\
  \dl(K) &= \dl(L) = \dl(L/M[d]) \\
  \dl(M/L) &= \dl(M/K) = \dl(t^dM/K).
 \end{align*}
\end{lemma}
\begin{proof}
 Suppose that $L/K$ is balanced, so $L/K\simeq A_d\oplus A_d$, so
 $\bt(L/K)=d$.  Using Lemma~\ref{lem-beta} we see that
 $\bt(L)\geq d$.  Similarly, as $L/K$ is a submodule of $M/K$ we see
 that $\bt(M/K)\geq d$.  Using either the inclusion $L\to M$ or the
 quotient map $M\to M/K$ we can now deduce that $\bt(M)\geq d$.  Next,
 as $L/K\simeq A_d\oplus A_d$ we have $t^d.(L/K)=0$, so $t^dL\leq K$.
 On the other hand, $\dim(t^dL)=\dim(L)-\dim(L[d])$, and
 $\dim(L[d])=2d$ because $\bt(L)\geq d$, so $\dim(t^dL)=\dim(L)-2d$.
 Moreover, we also have $\dim(K)=\dim(L)-\dim(L/K)=\dim(L)-2d$, so
 $K=t^dL$.  This implies that $L\leq t^{-d}K=\{m\in M\st t^dm\in K\}$.
 As $K=t^dL\leq t^dM$ we see that $K=t^dt^{-d}K$.  This means that
 there is a short exact sequence $M[d]\xra{}t^{-d}K\xra{t^d}K$, which
 gives $\dim(t^{-d}K)=\dim(K)+2d=\dim(L)$, so $L=t^{-d}K$.  This
 proves that~(a) implies both~(b) and~(c).

 Suppose instead that we assume~(b).  As $K=t^dL$ we have
 $L/K=L/t^dL$, and $L/t^dL$ is balanced because $d\leq\bt(L)$, so~(a)
 holds.  Alternatively, suppose that we assume~(c).  As $L=t^{-d}K$ we
 have $L/K=t^{-d}K/K=(M/K)[d]$, and this is balanced because
 $d\leq\bt(M/K)$, so again~(a) holds.  We now see that~(a), (b)
 and~(c) are equivalent.
 
 Suppose that~(a), (b) and~(c) hold.  We have seen
 that $\bt(M)\geq d$, so $M\simeq A_p\oplus A_{p+q}$ for some
 $p\geq d$.  From this decomposition it is clear that
 $t^dM\simeq A_{p-d}\oplus A_{p-d+q}$, so $\dl(t^dM)=q=\dl(M)$.
 Moreover, multiplication by $t^d$ gives an isomorphism
 $M/M[d]\to t^dM$, so $\dl(M/M[d])=q$.  As $\bt(L)$ is also at least
 $d$, the same logic gives $\dl(L)=\dl(t^dL)=\dl(L/L[d])$.  Here
 $t^dL=K$.  Moreover, $L[d]$ is a subgroup of $M[d]$ with the same
 dimension, so it must be the same as $M[d]$.  We thus have
 $\dl(L)=\dl(K)=\dl(L/M[d])$.  Now consider the quotient $N=M/K$.  We
 have seen that $\bt(N)\geq d$, so again
 $\dl(N)=\dl(t^dN)=\dl(N/N[d])$.  Here $N[d]=(t^{-d}K)/K=L/K$, so
 $N/N[d]=M/L$.  We also have $t^dN=(K+t^dM)/K$, but $K=t^dL\leq t^dM$,
 so $t^dN=(t^dM)/K$.  Putting this together,
 we get $\dl(M/K)=\dl(t^dM/K)=\dl(M/L)$.
\end{proof}

\begin{lemma}\label{lem-triangle}
 Let $M$ be a thin $A$-module, and let $N$ be a submodule.  Then each
 of the numbers $\dl(M)$, $\dl(N)$ and $\dl(M/N)$ is at most the sum
 of the other two.
\end{lemma}
\begin{proof}
 If $N$ is balanced, then the claim is just that $\dl(M)=\dl(M/N)$,
 and this follows from Lemma~\ref{lem-gap}.  Similarly, if $M/N$ is
 balanced, then the claim is just that $\dl(N)=\dl(M)$, which again
 follows from Lemma~\ref{lem-gap}.

 Now consider the general case.  We have $N\simeq A_n\oplus A_{n+p}$
 for some $n,p\geq 0$.  This contains a balanced submodule
 $N_0\simeq A_n^2$, and the first special case shows that nothing will
 change if we replace $N$ by $N/N_0$ and $M$ by $M/N_0$.  We may thus
 assume that $N\simeq A_p$ (so $\dl(N)=p$).  Next, we have
 $M/N\simeq A_m\oplus A_{m+s}$ for some $m,s\geq 0$, which gives
 $(M/N)/t^m(M/N) = M/(t^mM+N) \simeq A_m^2$.  Using Lemma~\ref{lem-gap}
 (with $M:=M$, $L:=M$, $K:=t^mM+N$) we see that
 $N\leq t^mM$ and that nothing relevant changes if we replace $M$ by
 $t^mM$.  We may thus assume that $M/N\simeq A_s$ (so $\dl(M/N)=s$).
 We now have a short exact sequence $A_p\to M\to A_s$.  It is clear
 from this that $\max(p,s)\leq\eta(M) = \dl(M)\leq p+s=\dim(M)$, and thus that
 $2\max(p,s)-p-s\leq\dl(M)\leq p+s$.  A check of cases shows that
 $2\max(p,s)-p-s=|p-s|$, and the claim follows from this.
\end{proof}

\begin{lemma}\label{lem-one-step-a}
 Let $0=M_0<M_1<\dotsb<M_d$ be a chain of thin modules with $d>0$ and
 $\dim(M_i)=i$ for all $i$.  Then either $M_d\simeq A_d$ or there
 exists $i$ with $0\leq i-1<i+1\leq d$ such that $M_{i+1}/M_{i-1}$ is
 balanced and $M_{i-1}=tM_{i+1}$.
\end{lemma}
\begin{proof}
 It is automatic that $M_0\simeq A_0$ and $M_1\simeq A_1$. Let $i$ be
 maximal such that $M_i\simeq A_i$.  If $i=d$ then the first case
 applies.  Suppose instead that $i<d$, and choose an element $u$ such
 that $M_i=A_i.u$.  Note that we must have $M_{i-1}=tM_i=A_{i-1}.tu$,
 because this is the only submodule of codimension one.  Now $M_{i+1}$
 contains $A_i$ and has dimension $i+1$ and is not isomorphic to
 $A_{i+1}$, so it must be isomorphic to $A_1\oplus A_i$ instead.  In
 particular, we see that $M_{i+1}[1]\simeq A_1^2$, which is strictly
 larger than $M_i[1]=A_i.t^{i-1}u$.  If we choose
 $v\in M_{i+1}[1]\sm M_i[1]$ we find that $M_{i+1}=A_i.u\oplus A_1.v$.
 This gives
 \[ \frac{M_{i+1}}{M_{i-1}} =
     \frac{A_i.u\oplus A_1.v}{A_{i-1}.tu\oplus 0} = A_1u \oplus A_1v,
 \]
 which is balanced and is annihilated by $t$.
\end{proof}

\begin{lemma}\label{lem-one-step-b}
 Let $0=M_0<M_1<\dotsb<M_d$ be a chain of thin modules with $d>0$ and
 $\dim(M_i)=i$ for all $i$.
 \begin{itemize}
  \item[(a)] Suppose that $0\leq k\leq\dl(M_d)$.  Then there exists
   $i$ with $\dl(M_i)=k$ and $\dl(M_d/M_i)=\dl(M_d)-k$.
  \item[(b)] Suppose that $\dl(M_d)=1$ and $d>1$.  Then there is
   an index $i$ with $0<i<d$ such that either $M_i$ or $M_d/M_i$ is
   balanced.
 \end{itemize}
\end{lemma}
\begin{proof}
 We will argue by induction on $d$.  If $d=1$ then~(a) is clear
 and~(b) is vacuous.  Suppose instead that $d>1$.  If $M_d\simeq A_d$
 then~(b) is again vacuous, and we must have $A_j=M_d[j]=t^{d-j}M_d$
 for all $j$, so in claim~(a) we can take $i=k$.  This just leaves the
 main case where $d>1$ but $M_d$ is not cyclic.  Here the previous
 lemma gives $j$ with $0\leq j-1<j+1\leq d$ such that
 $M_{j+1}/M_{j-1}$ is balanced and $M_{j-1}=tM_{j+1}$.  Put $N_i=M_i$
 for $i<j$, so in particular $N_{j-1}=tM_{j+1}$; then put
 $N_i=tM_{i+2}$ for $j\leq i\leq d-2$.  Using Lemma~\ref{lem-gap} we
 see that
 \[ \dl(N_m/N_i) = \begin{cases}
     \dl(M_m/M_i) & \text{ if } i\leq m < j \\
     \dl(M_{m+2}/M_i) & \text{ if } i\leq j\leq m  \\
     \dl(M_{m+2}/M_{i+2}) & \text{ if } j\leq i\leq m.
    \end{cases}
 \]
 In particular, we have $\dl(N_{d-2})=\dl(M_d)$.  For claim~(a) we can
 apply the induction hypothesis to the chain $(N_i)_{i=0}^{d-2}$; this
 gives an index $i$ such that $\dl(N_i)=k$ and
 $\dl(N_{d-2}/N_i)=\dl(M_d)-k$.  Put $i'=i$ if $i<j$, or $i'=i+2$ if
 $i\geq j$; we find that $\dl(M_{i'})=k$, and
 $\dl(M_d/M_{i'})=\dl(M_d)-k$, as required.

 Now suppose that $\dl(M_d)=1$ so claim~(b) is relevant.  Plainly $d \neq 2$.  If $d=3$
 then either $j=1$ and we can take $i=2$, or $j=2$ and we can take
 $i=1$.  Suppose instead that $d>3$, so the sequence $N_*$ is long
 enough to apply~(b) inductively, giving an index $i$ with $0<i<d-2$
 such that either $\dl(N_i)=0$ or $\dl(N_{d-2}/N_i)=0$.  We define
 $i'$ as before, and by checking the various possible cases we see
 that either $M_{i'}$ or $M_d/M_{i'}$ is balanced.
\end{proof}

\section{Graphs}
\label{sec-graphs}

To fix conventions we review some definitions related to graphs.  We
will use the word ``pregraph'' to refer to what most people call
graphs, so we can reserve the term ``graph'' for connected pregraphs
with specified bipartite structure, which is what we need for the bulk
of the paper.

\begin{definition}
 \begin{itemize}
  \item[(a)] A \emph{pregraph} is a pair $(V,E)$, where $V$ is a finite
   set, and $E$ is a subset of $V\tm V$ such that for all $u,v\in V$
   we have $(u,u)\not\in E$, and $(u,v)\in E$ iff $(v,u)\in E$.  The
   elements of $V$ are called \emph{vertices}, and the elements of $E$
   are called \emph{(directed) edges}.  We also write $\vrt(G)$ for
   $V$ and $\edg(G)$ for $E$.  Given an edge $e=(u,v)\in E$, we write
   $\be$ for the reversed edge $(v,u)$.  An \emph{undirected edge} is
   a set $\{u,v\}$ such that $(u,v)$ and $(v,u)$ are directed edges.
  \item[(b)] A \emph{morphism} from $G=(V,E)$ to $G'=(V',E')$
   will mean a function $f\:V\to V'$ such that $(f\tm f)(E)\sse E'$,
   so whenever there is an edge from $u$ to $v$ in $G$ there is also
   an edge from $f(u)$ to $f(v)$ in $G'$.  (In particular, this means
   that if $(u,v)$ is an edge in $G$ then $f(u)\neq f(v)$.)
  \item[(c)] A \emph{full subpregraph} of $G=(V,E)$ is a pregraph of the
   form $G'=(V',E\cap(V'\tm V'))$ for some subset $V'\sse V$.
  \item[(d)] Now consider an equivalence relation $\sim$ on $V$.
   Let $p\:V\to\bV=V/\sim$ be the associated quotient map, and put
   $\bE=(p\tm p)(E)$ and $\bG=(\bV,\bE)$.  We say that $\sim$ is
   \emph{prefoldable} if whenever $(u,v)\in E$ we have $u\not\sim v$.  If
   so, then $\bG$ is a pregraph and $p$ gives a morphism $G\to\bG$.  We
   say that $\bG$ is a \emph{prefolding} of $G$.  Equivalently, a
   prefolding is a pregraph morphism that is surjective on vertices
   and also surjective on edges.
 \end{itemize}
\end{definition}

\begin{definition}
 Let $G=(V,E)$ be a pregraph.
 \begin{itemize}
  \item[(a)] Let $u$ and $v$ be vertices of $G$.  A \emph{path of
    length $n$ from $u$ to $v$} is a list $u=w_0,w_0,\dotsc,w_n=v$ such
   that each pair $(w_i,w_{i+1})$ is an edge.  We write $d(u,v)$ for
   the minimum length of a path from $u$ to $v$, or $d(u,v)=\infty$ if
   there is no such path.  This gives a metric on $V$.
  \item[(b)] We say that $G$ is \emph{connected} if it is nonempty and
   $d(u,v)<\infty$ for all $u,v\in V$.
  \item[(c)] A \emph{cycle of length $n$} in $G$ is a list
   $c=(u_0,\dotsc,u_{n-1},u_n)$ of vertices such that $u_0=u_n$ and
   each pair $(u_i,u_{i+1})$ is an edge.  We say that $c$ is \emph{basic} if it
   has length two (and so has the form $(u,v,u)$ for some edge
   $(u,v)$).  We say that $c$ is \emph{nondegenerate}
   if $n>2$ and the vertices $u_0,\dotsc,u_{n-1}$ are all different.
  \item[(d)] We say that $G$ is a \emph{pretree} if it is connected
   and has no nondegenerate cycles.
 \end{itemize}
\end{definition}

\begin{definition}
 We write $\B$ for the graph $(\{0,1\},\{(0,1),(1,0)\})$ (which has
 two vertices connected by an edge).  A \emph{graph} is a connected
 pregraph $G$ equipped with a pregraph morphism $\pi\:G\to\B$ (which
 we call the \emph{parity map}).  We write $V_i=\pi^{-1}\{i\}$, so
 $V=V_0\amalg V_1$.  The vertices in $V_0$ are called \emph{even}, and
 those in $V_1$ are called \emph{odd}.  We also put $E_+=E\cap(V_0\tm
 V_1)$ and $E_-=E\cap(V_1\tm V_0)$.  As $\pi$ is a pregraph morphism,
 we see that $E=E_+\amalg E_-$.  Edges in $E_+$ are called
 \emph{positive}, and those in $E_-$ are called \emph{negative}.

 A graph morphism is a parity-preserving pregraph morphism.  We write
 $\Graphs$ for the category of graphs and graph morphisms.
\end{definition}

\begin{example}\label{eg-C-n}
 For any $n>0$ we have a pregraph $C(n)$ with vertex set $V=\Z/2n$
 and edge set
 \[ E = \{(i-1,i)\st i\in \Z/2n\} \cup \{(i+1,i)\st i\in\Z/2n\}. \]
 There is an evident parity map $\pi\:\Z/2n\to\Z/2=\{0,1\}$ making
 this a graph.  We put $e_i=(i-1,i)$ so that
 \[ E = \{ e_1,\dotsc,e_{2n},\be_1,\dotsc,\be_{2n} \}. \]
 We call $C(n)$ the \emph{$2n$-gon}.
\end{example}

\begin{remark}\label{rem-C-one}
 The case $n=1$ has some exceptional features.  For example, when
 $n=1$ we have $e_1=\be_2$ and $\be_1=e_2$ so $|\edg(C(1))|=2$,
 whereas $|\edg(C(n))|=4n$ for $n>1$.  However, this does not cause
 any real problems, and most of our more interesting results work
 uniformly for all $n$.
\end{remark}

\begin{remark}\label{rem-distance-parity}
 It is clear that any path from $u$ to $v$ has length congruent to
 $\pi(u)-\pi(v)$ mod two.  It follows that
 $d(u,v)=\pi(u)-\pi(v)\pmod{2}$.
\end{remark}

\begin{definition}\label{defn-foldable}
 We say that an equivalence relation on $V$ is \emph{foldable} if
 equivalent vertices always have the same parity.  Because every edge
 links vertices of opposite parity, we see that any foldable relation
 is prefoldable.  Moreover, the quotient $\bG=G/\sim$ has a unique
 parity map such that the quotient map $G\to\bG$ is a morphism of
 graphs.  We say that $\bG$ is a \emph{folding} of $G$.  A
 \emph{tree folding} of $G$ is a folding $G\to T$ where $T$ is a
 tree.

 We write $\Fold(G)$ for the set of foldings of $G$, and $\TFold(G)$
 for the set of tree foldings.  Note that if $G\xra{p_0}G_0$ and
 $G\xra{p_1}G_1$ are foldings, then there is at most one function
 $V_0\xra{q}V_1$ with $qp_0=p_1$, and if such a $q$ exists then it is
 a graph morphism and in fact a folding.  We give $\Fold(G)$ a partial
 order by declaring that $G_1\leq G_0$ if $q$ exists as above.  We can
 regard the parity map $\pi\:G\to\B$ as a tree folding, and it is
 the smallest element of $\Fold(G)$ or $\TFold(G)$.
\end{definition}

\section{Foldings from non-crossing matchings}

In this section we identify the set $I=\{1,2,\dotsc,2n\}$ with the set
of undirected edges in $C(n)$, letting $i\in I$ correspond to the edge
$\{i-1,i\}$.  We then write $\NCM(n)$ for the set of non-crossing
matchings on this set.

\begin{definition}\label{defn-T-tau}
 For each non-crossing matching $\tau\in\NCM(n)$ we let $\sim_\tau$ be
 the smallest equivalence relation on $\Z/2n=\vrt(C(n))$ such that
 $i\sim\tau(i)-1$ for all $i\in N_{2n}$.  This is foldable by
 Lemma~\ref{lem-NCM-parity}.  We write $p_\tau\:C(n)\to T_\tau$ for
 the corresponding folding.
\end{definition}

\begin{remark}\label{rem-T-tau}
 Note that if $\tau(i)=j$ then $\tau(j)=i$, so $p_\tau(i)=p_\tau(j-1)$
 and $p_\tau(j)=p_\tau(i-1)$.  This means that on edges we have
 $p_\tau(e_i)=p_\tau(\ov{e}_j)$ and $p_\tau(\ov{e}_i)=p_\tau(e_j)$.
\end{remark}

The following picture shows the non-crossing matching
$(1\;2)(3\;8)(4\;5)(6\;7)\in\NCM(4)$ and the associated folding.
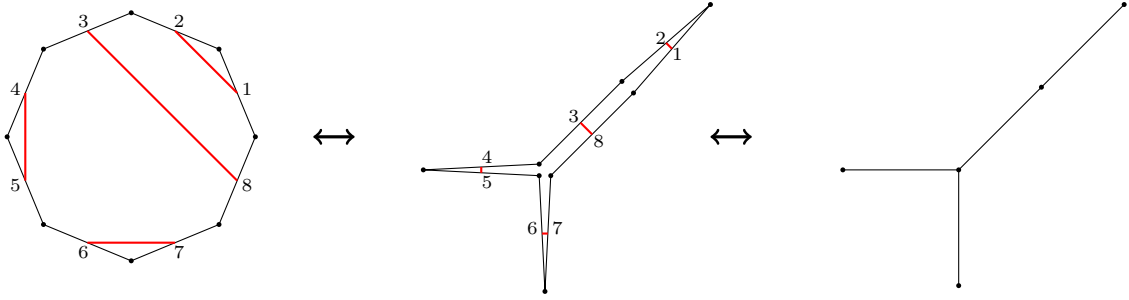
\begin{center}
 \begin{tikzpicture}[scale=1.1]
  \begin{scope}[scale=1.5]
   \def\r{0.924}
   \draw (0:1) -- (45:1) -- (90:1) -- (135:1) --
         (180:1) -- (225:1) -- (270:1) -- (315:1) -- (0:1);

   \foreach \i in {0,...,7} {
    \fill (45*\i : 1) circle(0.02);
   }

   \foreach \i in {1,...,8} {
    \draw (-22.5+45*\i :1.01) node{$\ss\i$};
   }

   \draw[thick,red] (22.5:\r) -- ( 67.5:\r);
   \draw[thick,red] (112.5:\r) -- (337.5:\r);
   \draw[thick,red] (157.5:\r) -- (202.5:\r);
   \draw[thick,red] (247.5:\r) -- (292.5:\r);
  \end{scope}
  \draw[very thick,<->] (2.2,0) -- (2.7,0);
  \begin{scope}[xshift=5cm,yshift=-0.4cm]
   \def\e{0.07}
   \draw (\e,-\e) -- (1+\e,1-\e) -- (2,2) -- (1-\e,1+\e) -- (-\e,\e) --
         (-1.4-\e,0) -- (-\e,-\e) -- (0,-1.4-\e) -- (\e,-\e);
   \fill (  \e, -\e) circle(0.03);
   \fill (1+\e,1-\e) circle(0.03);
   \fill (   2,   2) circle(0.03);
   \fill (1-\e,1+\e) circle(0.03);
   \fill ( -\e,  \e) circle(0.03);
   \fill (-1.4-\e,0) circle(0.03);
   \fill ( -\e, -\e) circle(0.03);
   \fill (0,-1.4-\e) circle(0.03);
   \draw[thick,red] (1.5-0.5*\e,1.5+0.5*\e) -- (1.5+0.5*\e,1.5-0.5*\e);
   \draw[thick,red] (0.5-\e,0.5+\e) -- (0.5+\e,0.5-\e);
   \draw[thick,red] (-0.7-\e,0.5*\e) -- (-0.7-\e,-0.5*\e);
   \draw[thick,red] (0.5*\e,-0.7-\e) -- (-0.5*\e,-0.7-\e);
   \draw ( 1.60, 1.40) node {$\ss 1$};
   \draw ( 1.40, 1.60) node {$\ss 2$};
   \draw ( 0.35, 0.65) node {$\ss 3$};
   \draw (-0.70, 0.15) node {$\ss 4$};
   \draw (-0.70,-0.15) node {$\ss 5$};
   \draw (-0.15,-0.70) node {$\ss 6$};
   \draw ( 0.15,-0.70) node {$\ss 7$};
   \draw ( 0.65, 0.35) node {$\ss 8$};
  \end{scope}
  \draw[very thick,<->] (7,0) -- (7.5,0);
  \begin{scope}[xshift=10cm,yshift=-0.4cm]
   \draw (0,0) -- (2,2);
   \draw (0,0) -- (-1.4,0);
   \draw (0,0) -- (0,-1.4);
   \fill (0,0) circle(0.03);
   \fill (1,1) circle(0.03);
   \fill (2,2) circle(0.03);
   \fill (-1.4,0) circle(0.03);
   \fill (0,-1.4) circle(0.03);
  \end{scope}
 \end{tikzpicture}
\end{center}

\begin{definition}\label{defn-TFold-n}
 We let $\TFold_n(C(n))$ be the set of all tree foldings of $C(n)$
 that have $n$ edges.
\end{definition}

\begin{proposition}\label{prop-max-foldings}
 The construction $\tau\mapsto p_\tau$ gives a bijection
 $\NCM(n)\to\TFold_n(C(n))$.
\end{proposition}

The proof uses the operation of deleting edges.  We spell out what we
mean by this as follows:
\begin{definition}\label{defn-delete}
 Let $G$ be a graph, and let $e=(u,v)$ be an edge in $G$.  As usual we
 write $\be$ for the reversed edge $(v,u)$.  We then let $G-e$ denote
 the pregraph with the same vertex set $\vrt(G-e)=\vrt(G)$, but
 $\edg(G-e)=\edg(G)\sm\{e,\be\}$.  We also use the same notation when
 deleting an undirected edge.
\end{definition}
\begin{remark}\label{rem-tree-delete}
 It is not hard to see that $G$ is a tree iff it is connected but
 $G-e$ is disconnected for all edges $e$.
\end{remark}

\begin{proof}[Proof of Proposition~\ref{prop-max-foldings}]
 We work throughout with undirected edges. Let $p\:C(n)\to T$ be an element of
 $\TFold_n(C(n))$. Consider an edge $e$ in $T$.  As $p$ is a folding
 (thus surjective on edges) we can choose
 an edge $e'=\{i-1,i\}\in C(n)$ such that $p(e')=e$.  The
 edges other than $e'$ in $C(n)$ give another path from $i-1$
 to $i$, so their images under $p$ connect $p(i-1)$ to $p(i)$.
 However, as $T$ is a tree, any path in $T$ from $p(i-1)$ to $p(i)$
 must involve $e$.  This means that there must be a second edge
 $e''\neq e'$ with $p(e'')=e$, so $|p^{-1}\{e\}|\geq 2$.  As
 $C(n)$ has precisely twice as many edges as $T$, we must have
 $|p^{-1}\{e\}|=2$ for all $e$.  It follows that there is a
 unique permutation $\tau$ of $I$ with $p\tau=p$ and $\tau(e')\neq e'$
 for all $e'$, and this satisfies $\tau^2=1$.

 Now suppose we have
 edges $e$ and $e'$ in $C(n)$ such that $e$, $e'$, $\tau(e)$ and
 $\tau(e')$ are distinct.  Put $C_1=C(n)-e-\tau(e)$ and
 $T_1=T-p(e)=T-p(\tau(e))$.  Then $C_1$ will split as $C_2\amalg C_3$,
 where $C_2$ and $C_3$ are connected, with $e'$ in $C_2$ say.  As $T$ is
 a tree we see that $T_1$ must be disconnected, so it will split as
 $p(C_2)\amalg p(C_3)$.  As $p(C_2)$ and $p(C_3)$ are disjoint but
 $e'$ is in $C_2$ with $p(e')=p(\tau(e'))$, we see that $\tau(e')$
 must also be in $C_2$, so $e$ and $e'$ cannot violate the
 non-crossing condition.  We conclude that $\tau\in\NCM(n)$.  This
 construction therefore gives a map $\phi\:\TFold_n(C(n))\to\NCM(n)$.

 Now suppose we start with $\tau\in\NCM(n)$ and we define
 $p_\tau\:C(n)\to T_\tau$ as in Definition~\ref{defn-T-tau}.  It is
 clear by construction that every undirected edge $e$ has
 $p_\tau(e)=p_\tau(\tau(e))$, and thus that $T_\tau$ has at most $n$
 undirected edges.  For any undirected edge $e$, the graph
 $C(n)-e-\tau(e)$ splits into two components, say $C_1\amalg C_2$.
 The non-crossing condition implies that no vertex in $C_1$ is
 equivalent to any vertex in $C_2$, so $T_\tau-p_\tau(e)$ splits as
 $p(C_1)\amalg p(C_2)$.  This means in particular that no other
 undirected edge in $C(n)$ can map to $p(e)$, so $p$ induces a
 bijection $\edg(C(n))/\ip{\tau}\to\edg(T_\tau)$, proving that
 $T_\tau$ has precisely $n$ undirected edges.  We also see that the
 deletion of any edge disconnects $T_\tau$, so $T_\tau$ is a tree.
 This construction therefore gives a map
 $\psi\:\NCM(n)\to\TFold_n(C(n))$ with
 $\psi(\tau)=(C(n)\xra{p_\tau}T_\tau)$.  We leave it to the reader to
 check that this is inverse to $\phi$.
\end{proof}

\section{Rings and spaces associated to graphs}
\label{sec-SY}

\begin{definition}\label{defn-YG}
 For any graph $G$, we put
 \[ Y_1(G) = \{m\:\edg(G)\to \CP^1 \st m(\be)=\chi(m(e))
                \text{ for all } e\in E\}.
 \]
 Given a tree folding $p\:G\to T$, we have an evident closed
 inclusion $p^*\:Y_1(T)\to Y_1(G)$.  We define $Y(G)$ to be the union
 of the subsets $p^*(Y_1(T))$ for all tree foldings of $G$.
\end{definition}

\begin{proposition}\label{prop-Y-functor}
 For any graph morphism $f\:G\to G'$, the resulting map
 $f^*\:Y_1(G')\to Y_1(G)$ sends $Y(G')$ to $Y(G)$.  Thus, $Y$ gives a
 contravariant functor from graphs to spaces.
\end{proposition}
\begin{proof}
 Consider a point $m'\in Y(G')$.  By definition, there must be a tree
 folding $p'\:G'\to T'$ such that $m'\in (p')^*Y_1(T')$, or
 equivalently $m'(u)=m'(v)$ whenever $p'(u)=p'(v)$.  Let $T$ be the
 full subpregraph of $T'$ with vertex set $p'(f(G))$.  Since $G$ is
 connected, we see that $T$ is also connected, and it inherits a
 parity map from $T'$ and is therefore a tree.  Let $j\:T\to T'$ be
 the inclusion, and let $p\:G\to T$ be the morphism such that
 $jp=p'f$.  For any edge $(u',v')$ in $T$ we can choose vertices
 $u,v\in G$ with $p(u)=u'$ and $p(v)=v'$.  We can then choose a path
 from $u$ to $v$ in $G$, and apply $f$ to it to get a path from $u'$
 to $v'$.  As $T$ is a tree, one of the edges in this path must be
 $(u',v')$.  This shows that the map $p\:G\to T$ is surjective on
 edges, and so is a folding.  We have $f^*(p')^*=p^*j^*\:Y_1(T')\to
 Y_1(G)$, so $f^*(m')\in p^*(Y_1(T)) \subseteq Y(G)$ as required.
\end{proof}

\begin{proposition}\label{prop-YCn}
 There is a homeomorphism $\sg\:X(n)\to Y(C(n))$ given by
 \begin{align*}
  \sg(\un{W})(e_i)   &= \om(W_i \ominus W_{i-1}) \\
  \sg(\un{W})(\be_i) &= \chi\om(W_i \ominus W_{i-1}).
 \end{align*}
\end{proposition}
\begin{proof}
 In view of Proposition~\ref{prop-max-foldings}, this is just a
 translation of Proposition~\ref{prop-X-prime}.
\end{proof}

\begin{definition}\label{defn-SG}
 We again consider a graph $G$, and we let $S_0(G)$ be the
 polynomial ring over $\Z$ with one generator $x_e = x_{uv}$ for each edge
 $e = (u, v)\in\edg(G)$.  We define a grading on this ring by $|x_e|=2$.
 For any cycle $c=(u_0,u_1,\dotsc,u_n=u_0)$ we put
 \[ r_c(t) = \prod_{i=0}^{n-1} (1 + t x_{u_i,u_{i+1}})
     \in S_0(G)[t]
 \]
 and $r^*_c(t)=r_c(t)-1$.  We let $S(G)$ denote the quotient of
 $S_0(G)$ by the ideal generated by coefficients of $r^*_c(t)$ for all
 cycles $c$ in $G$.

 For any graph morphism $f\:G\to G'$ we have a graded ring
 homomorphism $f_*\:S_0(G)\to S_0(G')$ given by
 $f_*(x_{uv})=x_{f(u),f(v)}$.  As $f$ sends edges to edges it also
 sends cycles to cycles and it follows easily that $f_*$ induces a
 homomorphism $S(G)\to S(G')$.  Thus, $S$ gives a functor
 $\Graphs\to\CRings_*$, where $\CRings_*$ is the category of
 commutative graded rings.
\end{definition}

\begin{remark}\label{rem-basic-nondeg}
 Consider a basic cycle $c=(u,v,u)$ in $G$.  We then have
 $r_c^*(t)=(x_{uv}+x_{vu})t+x_{uv}x_{vu}t^2$, so in $S(G)$ we have
 $x_{vu}=-x_{uv}$ and $x_{uv}^2=x_{vu}^2=x_{uv}x_{vu}=0$.  In
 particular, the generators $x_e$ for $e\in E_-$ are the negatives of
 the generators for $E_+$.  We refer to these relations as the
 \emph{basic relations}.  Put
 \[ S_1(G) = S_0(G)/(\text{basic relations}) =
     \frac{\Z[x_e\st e\in E_+]}{(x_e^2\st e\in E_+)}.
 \]
 The relations arising from nondegenerate cycles will be called
 \emph{nondegenerate relations}.
\end{remark}

\begin{remark}\label{rem-HY-one}
 Let $x$ denote the usual generator of $H^2(\CP^1)$.  For any edge
 $e\in\edg(G)$ we have a projection $\pi_e\:Y_1(G)\to\CP^1$ and we
 define $x_e=\pi_e^*(x)\in H^2(Y_1(G))$.  As the map $\chi$
 satisfies $\chi^*(x)=-x$, and $x^2=0$, we see that $x_{\be}+x_e=0$ and
 $x_e^2=0$.  Using this we obtain a natural isomorphism
 $\psi_1\:S_1(G)\to H^*(Y_1(G))$.
\end{remark}

\begin{remark}\label{rem-Y-zero}
 We could also define $Y_0(G)=\Map(\edg(G),\CP^\infty)$ and we would
 then have $H^*(Y_0(G))=S_0(G)$ but as far as we know this is not
 useful.
\end{remark}

\begin{remark}\label{rem-cycle-rot}
 There is an evident rotation operation on cycles, given by
 \[ \rho(u_0,u_1,\dotsc,u_{n-1},u_0) =
     (u_1,u_2,\dotsc,u_{n-1},u_0,u_1).
 \]
 We make the convention that indices are read modulo $n$ by default,
 so this can be written as $\rho(u)_i=u_{i+1}$.  It is clear that
 $r_{\rho(c)}(t)=r_c(t)$ and thus $r^*_{\rho(c)}(t)=r^*_c(t)$, so
 $\rho(c)$ gives the same relations as $c$.
\end{remark}

\begin{lemma}\label{rem-min-relations}
 The basic relations and the nondegenerate relations imply all
 relations in $S(G)$.
\end{lemma}
\begin{proof}
 Let $I$ be the ideal generated by the basic relations and the
 nondegenerate relations, and put $Q=S_0(G)/I$.  The claim is that
 $Q=S(G)$, or equivalently that $r_c(t)=1$ in $Q[t]$ for all cycles
 $c$.  We prove this do by induction on the length $n$ of $c$.  If
 $n\leq 2$ or $c$ is nondegenerate then the claim is clear.  If $n>2$
 and $c$ is degenerate then there exist $p,q$ with $0\leq p<q<n$ with
 $u_p=u_q$.  Put $c'=(u_p,u_{p+1},\dotsc,u_q)$ and
 $c''=(u_q,u_{q+1},\dotsc,u_{p+n})$, where the indices are considered
 modulo $n$ as usual.  Now $c'$ and $c''$ are shorter than $c$, so by
 induction we have $r_{c'}(t)=r_{c''}(t)=1$ in $Q[t]$.  It is also
 clear that $r_c(t)=r_{c'}(t)r_{c''}(t)$, so $r_c(t)=1$ in $Q[t]$ as
 required.
\end{proof}

\begin{remark}\label{rem-Y-tree}
 If $T$ is a tree then there are no nondegenerate relations and so
 $S(T)=S_1(T)$.  It is also clear that in this case we have
 $Y(T)=Y_1(T)$ and so $H^*(Y(T))=H^*(Y_1(T))=S_1(T)=S(T)$.

 For $C(n)$ there is essentially only one nondegenerate cycle, and the
 nondegenerate relations are the defining relations for $R(n)$, so
 $S(C(n))=R(n)$.
\end{remark}

\begin{proposition}\label{prop-S-HY}
 The map $\psi_1\:S_1(G)\to H^*(Y_1(G))$ induces a natural
 transformation $\psi\:S(G)\to H^*(Y(G))$, which is an isomorphism
 when $G$ is a tree, and agrees with $\phi\:R(n)\to H^*(X(n))$ when
 $G=C(n)$.
\end{proposition}
\begin{proof}
 The main point to check is that the composite
 \[ S_1(G)\to H^*(Y_1(G)) \xra{\res} H^*(Y(G)) \]
 sends $r_c(t)$ to $1$ for any cycle $c$ in $G$.  Our bipartite
 structure implies that $c$ must have length $2n$ for some $n$, and
 this means that it corresponds to a graph morphism $f\:C(n)\to G$.
 Let $c_0$ be the obvious cycle $(0,1,2,\dotsc,2n=0)$ in $C(n)$.  We
 observed in the introduction that symmetric functions in the elements
 $x_1,\dotsc,x_{2n}\in H^2(X(n))$ are all zero.  Using our
 homeomorphism $X(n)\simeq Y(C(n))$ we deduce that $r_{c_0}(t)=1$ in
 $H^*(Y(C(n)))[t]$.  Using the naturality square
 \[ \xymatrix{
  S(C(n)) \ar[r] \ar[d] & H^*(Y(C(n))) \ar[d] \\
  S(G) \ar[r] & H^*(Y(G))
 } \]
 we deduce that $r_c(t)=1$ in $H^*(Y(G))[t]$.  It follows that the map
 $S_1(G)\to H^*(Y(G))$ factors uniquely through a map
 $S(G)\to H^*(Y(G))$.  The remaining claims are now clear.
\end{proof}

\begin{proposition}\label{prop-graph-union}
 Suppose that $G$ is the union of two connected full subgraphs $G_0$
 and $G_1$ together with an edge $e=(a_0,a_1)$, such that
 $G_0\cap G_1=\emptyset$ but $a_i\in G_i$ for $i=0,1$.  Then
 \begin{align*}
  S(G) &= S(G_0)\ot S(G_1)\ot \Z[x_e]/(x_e^2) \\
  Y(G) &= Y(G_0)\tm Y(G_1)\tm \CP^1.
 \end{align*}
 Thus, if the rings $S(G_i)$ are torsion-free and the maps
 $\psi\:S(G_i)\to H^*(Y(G_i))$ are isomorphisms, then $S(G)$ is also
 torsion-free and $\psi\:S(G)\to H^*(Y(G))$ is an isomorphism.
\end{proposition}
\begin{proof}
 First, it is clear that a nondegenerate cycle in $G$ cannot involve
 $e$ or $\ov{e}$, so it either lives wholly in $G_0$ or wholly in
 $G_1$.  The description of $S(G)$ is clear from this.  Next, there is
 an evident way to identify $Y_1(G)$ with
 $Y_1(G_0)\tm Y_1(G_1)\tm\CP^1$.  As $Y$ is functorial with respect to
 the inclusions $G_i\to G$, we see that
 $Y(G)\sse Y(G_0)\tm Y(G_1)\tm\CP^1$.  In the opposite direction,
 suppose we have $m_i\in Y(G_i)\sse Y_1(G_i)$ for $i=0,1$ and
 $n\in\CP^1$.  Then there must exist tree foldings $p_i\:G_i\to T_i$
 and elements $\ov{m}_i\in Y_1(T_i)$ with $m_i=p_i^*(\ov{m}_i)$.  Let
 $T$ be obtained from $T_0\amalg T_1$ by adding an edge from
 $p_0(a_0)$ to $p_1(a_1)$.  There is an evident tree folding
 $p\:G\to T$, and an evident way to interpret the triple
 $(\ov{m}_0,\ov{m}_1,n)$ as an element of $Y_1(T)$ with
 $p^*(\ov{m}_0,\ov{m}_1,n)=(m_0,m_1,n)$.  This shows that
 $Y(G_0)\tm Y(G_1)\tm\CP^1=Y(G)$.  It is clear by construction that
 we have a commutative diagram
 \[ \xymatrix{
  S(G_0)\ot S(G_1)\ot\left(\Z[x_e]/(x_e^2)\right)
   \ar[d]_{\simeq} \ar[rrr]^{\psi\ot\psi\ot 1}
   &&&
  H^*(Y(G_0))\ot H^*(Y(G_1)) \ot H^*(\CP^1) \ar[d] \\
  S(G) \ar[rrr]_{\psi} &&&
  H^*(Y(G))
 } \]
 If the rings $S(G_i)$ are torsion free and the maps
 $\psi\:S(G_i)\to H^*(Y(G_i))$ are isomorphisms then the top map will
 be an isomorphism by assumption and the right hand map will be an
 isomorphism by the K\"unneth Theorem, so the map
 $\psi\:S(G)\to H^*(Y(G))$ will also be an isomorphism by chasing the
 diagram.
\end{proof}

\begin{corollary}\label{cor-add-edge}
 Suppose that $G$ is obtained from $G_0$ by attaching a single extra
 vertex and a single edge $e$ connecting that vertex to one of the old
 vertices.  Then $S(G)=S(G_0)\ot\Z[x_e]/(x_e^2)$ and
 $Y(G)=Y(G_0)\tm\CP^1$.
\end{corollary}
\begin{proof}
 This is the case of the Proposition where $G_1$ has a single vertex
 and no edges.
\end{proof}

We next discuss how to give $Y(G)$ the structure of an ordered
simplicial complex.  (The ordering is necessary because the standard
recipe for triangulating the product of two simplicial complexes
requires an ordering of each factor.)

\begin{construction}\label{cons-triangulation}
 We let $P$ denote the set $\{-3,-2,-1,1,2,3\}$ equipped with the
 partial ordering where $u<v$ iff $|u|<|v|$ in $\N$.  We define
 $\chi\:P\to P$ by $\chi(u)=-u$; this is an automorphism of partially
 ordered sets.  As usual we can regard $P$ as an abstract simplicial
 complex, where the elements of $P$ are the vertices, and the nonempty
 chains are the simplices.  We define $f_0\:|P|\to\R^3$ by putting
 $f_0(\pm 1)=(\pm 1,0,0)$ and $f_0(\pm 2)=(0,\pm 1,0)$ and
 $f_0(\pm 3)=(0,0,\pm 1)$, and then extending linearly over
 simplices.  This gives a homeomorphism from $|P|$ to an octahedron.
 We can thus define a homeomorphism $f\:|P|\to S^2$ by
 $f(a)=f_0(a)/\|f_0(a)\|$.  By construction we have
 $\chi\circ f=f\circ|\chi|$.

 Now put
 \[ YP(G)=
     \{m\:\edg(G)\to P\st m(\ov{e})=\chi(m(e)) \text{ for all } e\}.
 \]
 It is a standard fact that geometric realisation of posets preserves
 finite limits, so $f$ induces a natural homeomorphism
 $|YP(G)|\to Y(G)$.  In particular, any folding $p\:G\to G'$ embeds
 $Y(G')$ as a subcomplex of $Y(G)$.
\end{construction}

\section{Flags give trees}
\label{sec-flags-give-trees}

Throughout this section we fix a flag $\uW\in X(n)$.  We know from
Proposition~\ref{prop-YCn} that there must exist a tree folding
$p\:C(n)\to T$ such that $\bom(\un{W})\in p^*(Y(T))\sse Y(C(n))$.  In
this section we will construct a canonical choice of such a folding,
and show that it is intimately connected with the algebraic structure
of the flag.

\begin{definition}
 We let $\tW$ denote the unrolled flag corresponding to $\un{W}$, as
 in Definition~\ref{defn-unroll}.  For $m,k\in\Z$ with $m\leq k$ we
 put $d_{\uW}(k,m)=d_{\uW}(m,k)=\dl(\tW_k/\tW_m)$ (where $\dl$ is the
 imbalance, as in Definition~\ref{defn-imbalance}).
\end{definition}

Recall that a \emph{pseudometric} on a set $I$ is a function
$d\:I\tm I\to\R$ that satisfies the axioms for a metric except that
$d(i,j)$ may be zero even when $i\neq j$.

\begin{proposition}\label{prop-metric}
 The function $d$ is a pseudometric on $\Z$ with $d(i,j)=d(i',j')$
 whenever $i=i'\pmod{2n}$ and $j=j'\pmod{2n}$.  We also have
 $d(i,j)=j-i\pmod{2}$.
\end{proposition}

\begin{proof}[Proof of Proposition~\ref{prop-metric}]
 It is clear from the definitions that $d(i,j)\geq 0$ and
 $d(i,j)=d(j,i)$.  For $i\leq j\leq k$ we have a short exact sequence
 of thin modules $\tW_j/\tW_i\to\tW_k/\tW_i\to\tW_k/\tW_j$.
 Lemma~\ref{lem-triangle} therefore tells us that each of the numbers
 $d(i,j)$, $d(j,k)$ and $d(i,k)$ is at most the sum of the other two.
 This gives all possible triangle inequalities involving $i$, $j$ and
 $k$.  From the definitions we have $\tW_i/\tW_{i-2n}=\tW_i/t^n\tW_i$,
 which is a thin module of exponent $n$ and dimension $2n$, so it is
 balanced, so $d(i,i-2n)=0$.  It follows from this and the triangle
 inequality that $d(i,j)=0$ whenever $i=j\pmod{2n}$, or more generally
 that $d(i,j)=d(i',j')$ whenever $i=i'\pmod{2n}$ and $j=j'\pmod{2n}$.

 It is clear from the definitions that $\dl(M)=\dim(M)\pmod{2}$, which
 in the present case gives $d(i,j)=j-i\pmod{2}$.
\end{proof}

It follows from the Proposition that $d$ induces a pseudometric on the
set $\Z/2n=\vrt(C(n))$, which we again denote by $d$.  We define an
equivalence relation on this set by $u\sim v$ iff $d(u,v)=0$.  The
congruence $d(i,j)=j-i\pmod{2}$ shows that this is foldable.  We write
$q\:C(n)\to T$ for the corresponding folding.  We also write $d$
again for the metric on $\vrt(T)$ induced by $d$.  On the other hand,
the graph structure of $T$ gives a path-length metric on $\vrt(T)$,
which we temporarily denote by $d^*$.  We will show below that
$d^*=d$.

\begin{lemma}\label{lem-d-dstar}
 Consider elements $a,b\in\vrt(T)$.  Then the following are
 equivalent:
 \begin{itemize}
  \item[(a)] $(a,b)\in\edg(T)$
  \item[(b)] There are integers $i,j\in\Z$ representing $a$ and $b$
   with $|i-j|=1$
  \item[(c)] $d(a,b)=1$.
 \end{itemize}
\end{lemma}
\begin{proof}
 The equivalence of~(a) and~(b) is the definition of $T$, and is just
 recorded for ease of reference.  It is clear that any $1$-dimensional
 $A$-module $M$ has $\dl(M)=1$, and from this it follows that~(b)
 implies~(c).  Now suppose that~(c) holds.  Choose integers $i$ and
 $j$ representing $a$ and $b$ with $|i-j|$ as small as possible.
 After exchanging $a$ and $b$ if necessary, we may assume that $i<j$.
 If $j-i>1$ we can apply Lemma~\ref{lem-one-step-b}(b) to the modules
 $M_t/M_i$ to get an index $k$ with $i<k<j$ such that either $k\sim i$
 or $k\sim j$, which contradicts our choice of $i$ and $j$.  We must
 thus have $|i-j|=1$ as required.
\end{proof}

\begin{lemma}\label{lem-treelike}
 Suppose that $a,b,x,y\in\vrt(T)$ with $d''(a,x)=d''(a,y)$ and
 $d''(x,b)=d''(y,b)$ and
 $d''(a,x)+d''(x,b)=d''(a,y)+d''(y,b)=d''(a,b)$.  Then $x=y$.
\end{lemma}
\begin{proof}
 We first choose numbers $a,b,x,y\in\Z$ representing the specified
 equivalence classes in $\vrt(T)$.  As $i\sim(i+2n)$ for
 all $i$, these lifts can be chosen in such a way that
 $a\leq x\leq y\leq b$.  Put $K=\tW_x/\tW_a$ and $L=\tW_y/\tW_a$ and
 $M=\tW_b/\tW_a$.  The hypotheses now say that there are natural
 numbers $q,s\geq 0$ such that $\dl(K)=\dl(L)=q$ and
 $\dl(M/K)=\dl(M/L)=s$ and $\dl(M)=q+s$.  We must show that
 $\dl(L/K)=0$.  We have $K\simeq A_i\oplus A_{i+q}$ for some $i$.
 This contains a balanced submodule $K[i]\simeq A_i\oplus A_i$, and
 Lemma~\ref{lem-gap} implies that nothing relevant will change if
 we take quotients by $K[i]$ everywhere.  We may thus assume that
 $K\simeq A_q$.  Next, we have $M/L\simeq A_j\oplus A_{j+s}$ for some
 $j\geq 0$.  It follows that the quotient $M/t^jM\simeq A_j\oplus A_j$
 is balanced, and that $L\leq t^jM$.  Again, nothing relevant will
 change if we replace $M$ by $t^jM$, so we may assume that
 $M/L\simeq A_s$.  By assumption we have $\dl(M)=q+s$, so we can
 choose generators $u,v$ for $M$ with $M=A_mu\oplus A_{m+q+s}v$ for
 some $m\geq 0$.  This means that $\dim(M)=2m+q+s$ and
 $\dim(M/L)=\dim(A_s)=s$, so $\dim(L)=2m+q$.  We also have $\dl(L)=q$
 by assumption, so we must have $L\simeq A_m\oplus A_{m+q}$.  Now
 $L[m]\simeq A_m^2\simeq M[m]$ and $L[m]\leq M[m]$ so (by comparing
 dimensions) we must have $L[m]=M[m]$.  Now $L/L[m]$ is isomorphic to
 $A_s$, and this is a submodule of $M/M[m]=A_{q+s}v$.  The only
 submodule of $A_{q+s}v$ isomorphic to $A_q$ is
 $A_q.t^sv=t^s.(M/M[m])$.  From this we conclude that
 \[ L = M[m] + t^s M = A_m.u \oplus A_{m+q}.t^sv.  \]
 Similarly, $K$ is also uniquely determined by the numerical
 invariants.  More precisely, as $\dim(K)=q$ and $\dim(M)=2m+q+s$ we
 have $\dim(M/K)=2m+s$.  As $\dl(M/K)=s$ we also have
 $M/K\simeq A_m\oplus A_{m+s}$.  This means that
 \[ \frac{M}{t^mM+K} \simeq
     \frac{A_m\oplus A_{m+s}}{t^m(A_m\oplus A_{m+s})}
     \simeq A_m\oplus A_m \simeq \frac{M}{t^mM}.
 \]
 As $t^mM\leq t^mM+K$ we can just count dimensions to see that
 $K\leq t^mM=A_{q+s}.t^mv$.  There is only one submodule of $A_{q+s}$
 isomorphic to $A_q$, so we conclude that $K=A_q.t^{m+s}v$.  We now
 have
 \[ \frac{L}{K} =
     \frac{A_mu\oplus A_{m+q}.t^sv}{0\oplus A_q.t^{m+s}v} =
      A_mu \oplus A_mt^sv \simeq A_m^2.
 \]
 We thus have $\dl(L/K)=0$ as required, so $x$ and $y$ represent the
 same point in $\vrt(T)$.
\end{proof}

\begin{proposition}
 The graph $T$ is a tree, and $d^*=d$.
\end{proposition}
\begin{proof}
 It is clear by construction that $d(a,b)=0$ iff $a=b$ iff
 $d^*(a,b)=0$.  Using Lemma~\ref{lem-d-dstar} we also see that
 $d(a,b)=1$ iff there is an edge from $a$ to $b$ iff $d^*(a,b)=1$.
 Using the triangle inequality for $d$ we deduce that
 $d(a,b)\leq d^*(a,b)$ for all $a$ and $b$.

 Next, we can choose representing integers $a,b\in\Z$ with $a\leq b$,
 and put $r=d(a,b)$.  We can then apply Lemma~\ref{lem-one-step-b} to
 the chain of modules $\tW_i/\tW_a$ to see that there exists $x$ with
 $d(a,x)=1$ and $d(x,b)=r-1$.  By an evident inductive extension of
 this, we can choose $a=x_0,x_1,\dotsc,x_r=b$ with $d(x_i,x_{i+1})=1$
 and $d(x_{i+1},b)=r-i-1$ for all $i$, showing that $d^*(a,b)=d(a,b)$.

 Now suppose that $T$ is not a tree.  This means that $T$ contains a
 cycle $c$ of minimal length.  As every edge in $T$ links an odd
 vertex to an even vertex, we see that the length must be even, say
 $c=(c_0,c_1,\dotsc,c_{2p}=c_0)$.  As $c$ is minimal, the shortest
 path in $T$ from $c_i$ to $c_j$ must be the same as the shortest path
 in $C(n)$.  We now see that $d(c_0,c_p)=p$ and
 $d(c_p,c_{p+1})=d(c_p,c_{p-1})=1$ and
 $d(c_{p-1},c_0)=d(c_{p+1},c_0)=p-1$.  This contradicts
 Lemma~\ref{lem-treelike}.
\end{proof}

\begin{proposition}
 Let $0<i<j\leq 2n$.
 \begin{itemize}
  \item[(a)] If $d(i-1,j-1)=d(i,j)=0$ then
   $\om(W_i\ominus W_{i-1})=\om(W_j\ominus W_{j-1})$.
  \item[(b)] If $d(i-1,j)=d(i,j-1)=0$ then
   $\om(W_i\ominus W_{i-1})=\chi(\om(W_j\ominus W_{j-1}))$.
 \end{itemize}
\end{proposition}
\begin{proof}
 \begin{itemize}
  \item[(a)] In this case the assumption is that $W_{j-1}/W_{i-1}$ and
   $W_j/W_i$ are balanced.  This implies that $j-i$ must be even, say
   $j-i=2d$.  We then have $W_{j-1}/W_{i-1}\simeq W_j/W_i\simeq A_d^2$,
   and it follows using Lemma~\ref{lem-gap} that
   $V(n)[d]\leq W_{j-1}<W_j$ and $W_{i-1}=t^dW_{j-1}$ and
   $W_i=t^dW_j$.  Choose a nonzero element $a\in W_j\ominus W_{j-1}$.
   We can write this as $a=\sum_{k=0}^{n-1}a_kt^k$ for some elements
   $a_k\in\C^2$.  This is orthogonal to $W_{j-1}$ and thus to
   $V(n)[d]$, which implies that $a_k=0$ for $k\geq n-d$.  It follows
   that $t^da\neq 0$ and $\om(t^da)=\om(a)$.  We also claim that
   $t^da\in W_i\ominus W_{i-1}$.  Indeed, it is clear that
   $t^da\in t^dW_j=W_i$.  For $b'\in W_{i-1}$ we can choose
   $b\in W_{j-1}$ with $b'=t^db$.  We then have
   $\ip{t^da,b'}=\ip{t^da,t^db}=\ip{a,(t^*)^dt^db}$.  However, for any
   $x\in V(n)$ it is easy to see that $(t^*)^dt^dx\in x+V(n)[d]$.  As
   $V(n)[d]\leq W_{j-1}$ and $\ip{a,W_{j-1}}=0$ we conclude that
   $\ip{t^da,b'}=0$, as required.  It now follows that
   \[ \om(W_j\ominus W_{j-1}) = \C.\om(t^da) = \C.\om(a)
        = \om(W_i\ominus W_{i-1}).
   \]
  \item[(b)] In this case we have $W_{j-1}/W_i\simeq A_d^2$ and
   $W_j/W_{i-1}\simeq A_{d+1}^2$ for some $d\geq 0$.  This implies
   that $W_i=t^dW_{j-1}$ and $W_{i-1}=t^{d+1}W_j$.  Put
   $U=t^dW_j=t^{-1}W_{i-1}$, so $W_{i-1}<W_i=t^dW_{j-1}<U$ and
   $U/W_{i-1}\simeq A_1^2$.  By the same argument as in case~(a) we
   have
   \[ \om(W_j\ominus W_{j-1}) =
      \om(t^dW_j\ominus t^dW_{j-1}) =
      \om(U\ominus W_i).
   \]
   Proposition~\ref{prop-isometry} tells us that the map
   $\om\:U\ominus W_{i-1}\to\C^2$ is an isometric isomorphism, so in
   particular it preserves orthogonal complements.  It follows that
   $\om(U\ominus W_i)=\chi(W_i\ominus W_{i-1})$ as claimed.
 \end{itemize}
\end{proof}

\section{Hedgehog foldings}
\label{sec-hedgehog}

We next describe a special class of foldings of $C(n)$ which will be
used to prove that the map $R(n)\to H^*(X(n))$ is an isomorphism.

\begin{definition}\label{defn-Ln}
 We let $L(n)$ denote the linear graph with vertex set
 $\{0,\dotsc,2n\}$, and an edge from $i$ to $j$ iff $|i-j|=1$.
 As with $C(n)$, we denote the edge $(i-1, i)$ by $e_i$.
\end{definition}

Note that $C(n)$ can be regarded as a folding of $L(n)$ where $0$ is
identified with $2n$.  We will primarily be interested in foldings of
$C(n)$, but to handle various special and degenerate cases it is
convenient to work with $L(n)$ as well.

\begin{definition}
 For any subset $A\sse\{1,2,\dotsc,2n-1\}$ we let $\sim_A$ denote the
 smallest equivalence relation on $\vrt(L(n))$ such that
 $(i-1)\sim_A(i+1)$ for all $i\in A$.  This is clearly foldable, and
 we write $\tp_A\:L(n)\to\tH(A)$ for the resulting folding.  This
 induces a folding $p_A\:C(n)\to H(A)$, where $H(A)$ is obtained from
 $\tH(A)$ by identifying $0$ with $2n$.

 Foldings arising this way are called \emph{hedgehog foldings} of
 $L(n)$ or $C(n)$.  The points in $A$ are called \emph{pinch points}.
\end{definition}

\begin{remark}\label{rem-YHA-intersection}
 We note that $Y(H(A))=\bigcap_{i\in A}X(n,i)$.
\end{remark}

\begin{proposition}\label{prop-SHA}
 For any $A$ as above we have
 \[ S(H(A)) = R(n)/(x_i+x_{i+1}\st i\in A). \]
 Moreover, if the maps $\phi\:R(m)\to H^*(X(m))$ are isomorphisms for
 all $m<n$, then the maps $\psi\:S(H(A))\to H^*(Y(H(A)))$ are
 isomorphisms for all $A\sse\{1,2,\dotsc,2n-1\}$ with
 $A\neq\emptyset$.
\end{proposition}
The rest of this section will form the proof of this proposition.

We can analyse the structure of $\tH(A)$ and $H(A)$ as follows.  Put
\[ A^\# = \{i\in\{0,\dotsc,2n\}\st i-1\not\in A\}, \]
so $0,1\in A^\#$ and $|A^\#|=2n+1-|A|$.  Note that every vertex in
$L(n)$ comes from an integer in $i\in\{0,\dotsc,2n\}$.  If $i\not\in A^\#$
then $i\sim_A i-2$.  Using this rule repeatedly we find that
\[ A^\# = \{\min(\tp_A^{-1}\{v\})\st v\in\vrt(\tH(A))\} \]
and thus that $\tp_A$ restricts to give a bijection
$A^\#\to\vrt(\tH(A))$.  We use this to identify $A^\#$ with
$\vrt(\tH(A))$.

We now list the elements of $A^\#$ as $i_0,\dotsc,i_r$, so $i_0=0$ and
$i_1=1$.  We take $i_{r+1}$ to be $2n+1$.  For $1\leq t\leq r$ we let
$d_t$ denote the edge $\tp_A(e_{i_t})\in\edg(\tH(A))$.  Note that if
$i\not\in A^\#$ then $i\sim_A(i-2)$ so $\tp_A(e_i)=\tp_A(\ov{e}_{i-1})$.
It follows that if $i_t\leq i<i_{t+1}$ then $\tp_A(e_i)$ is the same as
$d_t$ (if $i-i_t$ is even) or $\ov{d}_t$ (if $i-i_t$ is odd).  If
$i_{t+1}-i_t$ is even we say that $d_t$ is a \emph{spine} and $i_t$ is
a \emph{spine vertex}.  If $i_{t+1}-i_t$ is odd we say that $d_t$ is a
\emph{body edge} and $i_t$ is a \emph{body vertex}.  We also consider
$i_0=0$ to be a body vertex (but $d_0$ is not defined).  We let
$B\tH(A)$ denote the full subgraph of $\tH(A)$ on the body vertices,
and call this the \emph{body} of $\tH(A)$.

The following picture illustrates the case where $n=9$ and
\begin{align*}
 A    &= \{2,3,4,7,10,12,15,16\} \\
 A^\# &= \{0,1,2,6,7,9,10,12,14,15,18\}.
\end{align*}

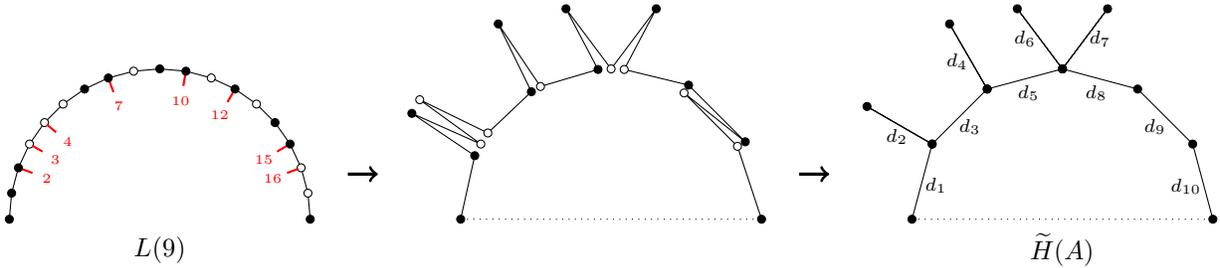
\begin{center}
 \begin{tikzpicture}[scale=2,
    vtxa/.style={circle,draw=black,fill=black,inner sep=0pt,minimum size=3pt},
    vtxb/.style={circle,draw=black,fill=white,inner sep=0pt,minimum size=3pt}
  ]
  \begin{scope}
   \foreach \i in {0,...,17} {
    \draw (10*\i:1) -- (10+10*\i:1);
   }
   \foreach \i in {2,3,4,7,10,12,15,16} {
    \draw[thick,red] (180-10*\i:0.98) -- (180-10*\i:0.90);
    \draw[red] (180-10*\i:0.8) node {$\scriptscriptstyle \i$};
    \draw (170-10*\i:1) node[vtxb] {};
   }
   \foreach \i in {0,1,2,6,7,9,10,12,14,15,18} {
    \draw (180-10*\i:1) node[vtxa] {};
   }
   \draw[very thick,->] (1.25,0.3) -- (1.45,0.3);
   \draw (0,-0.2) node{$L(9)$};
  \end{scope}
  \begin{scope}[xshift=3cm]
   \draw[white] (-0.1,1.7) -- (0.1,1.7);
   \draw[dotted] (-1,0) -- (1,0);
   \draw (180:1)    node[vtxa] {} --
         (155:1)    node[vtxa] {} --
         (152:1.5)  node[vtxa] {} --
         (150:1)    node[vtxb] {} --
         (148:1.5)  node[vtxb] {} --
         (145:1)    node[vtxb] {} --
         (122:1)    node[vtxa] {} --
         (120:1.5)  node[vtxa] {} --
         (118:1)    node[vtxb] {} --
         (95:1)     node[vtxa] {} --
         (-0.3,1.4) node[vtxa] {} --
         (0,1)      node[vtxb] {} --
         (0.3,1.4)  node[vtxa] {} --
         (85:1)     node[vtxb] {} --
         (60:1.03)  node[vtxa] {} --
         (30:1.03)  node[vtxa] {} --
         (60:0.97)  node[vtxb] {} --
         (30:0.97)  node[vtxb] {} --
         (0:1)      node[vtxa] {};
   \draw[very thick,->] (1.25,0.3) -- (1.45,0.3);
  \end{scope}
  \begin{scope}[xshift=6cm]
   \draw[dotted] (-1,0) -- (1,0);
   \draw (180:1)    node[vtxa] {} --
         (150:1)    node[vtxa] {} --
         (150:1.5)  node[vtxa] {} --
         (150:1)    node[vtxa] {} --
         (120:1)    node[vtxa] {} --
         (120:1.5)  node[vtxa] {} --
         (120:1)    node[vtxa] {} --
         (90:1)     node[vtxa] {} --
         (-0.3,1.4) node[vtxa] {} --
         (0,1)      node[vtxa] {} --
         (0.3,1.4)  node[vtxa] {} --
         (90:1)     node[vtxa] {} --
         (60:1)     node[vtxa] {} --
         (30:1)     node[vtxa] {} --
         (0:1)      node[vtxa] {};
   \draw (165:0.87)  node {$\scriptstyle d_1$};
   \draw (-1.1,0.55) node {$\scriptstyle d_2$};
   \draw (135:0.87)  node {$\scriptstyle d_3$};
   \draw (-0.7,1.05) node {$\scriptstyle d_4$};
   \draw (105:0.87)  node {$\scriptstyle d_5$};
   \draw (-0.25,1.2) node {$\scriptstyle d_6$};
   \draw ( 0.25,1.2) node {$\scriptstyle d_7$};
   \draw ( 75:0.87)  node {$\scriptstyle d_8$};
   \draw ( 45:0.87)  node {$\scriptstyle d_9$};
   \draw ( 15:0.85)  node {$\scriptstyle d_{10}$};
   \draw (0,-0.2) node{$\tH(A)$};
  \end{scope}
 \end{tikzpicture}
\end{center}
In the left hand picture, the radial lines indicate the elements of
$A$.  The solid circles indicate the points of $A^\#$, which become
the vertices of the right hand picture.  The empty circles indicate
the remaining vertices of $L(9)$, which become identified in $\tH(A)$
with the solid circles.  The edges $d_2$, $d_4$, $d_6$ and $d_7$ are
spines, whereas $d_1$, $d_3$, $d_5$, $d_8$, $d_9$ and $d_{10}$ form
the body.  It is clear in this case that the body is isomorphic to
$L(3)$.  We will see below that the body is always isomorphic to
$L(m)$ for some $m\leq n$.

\begin{lemma}\label{lem-body-edge}
 For all $t\in\{1,\dotsc,r\}$, the vertex $\tp_A(i_t-1)$ is the largest
 body vertex that is less than $i_t$.
\end{lemma}
\begin{proof}
 This is clear for $t=1$ because $i_1-1=0=i_0$ and this is a body
 vertex by definition.  Now suppose that $t>1$.  If $i_{t-1}$ is a
 body vertex then $i_t-i_{t-1}$ is odd, and by the definition of
 $A^\#$ we have $j\in A$ for $i_{t-1}\leq j\leq i_t-2$.  It follows
 that $\tp_A(i_t-1)=\tp_A(i_{t-1})$ (which we identify with $i_{t-1}$), so
 the claim holds.  Otherwise $i_t-i_{t-1}$ is even and a similar
 argument gives $\tp_A(i_t-1)=\tp_A(i_{t-1}-1)$; the claim therefore holds
 by induction on $t$.
\end{proof}

\begin{corollary}\label{cor-hedgehog-structure}
 The body $B\tH(A)$ is isomorphic to $L(m)$ for some $m\leq n$.  Each
 spine has one end in the body and one end outside the body.
\end{corollary}
\begin{proof}
 Let $j_0,\dotsc,j_q$ be the body vertices, so the upper end of the
 $k$'th body edge is $j_k$.  It follows from the Lemma that the lower
 end is $j_{k-1}$.  Note also that
 \[ \sum_{t=1}^{r} (i_{t+1}-i_t) = i_{r+1}-i_1 = (2n+1)-1 = 2n, \]
 so the number of odd summands must be even.  There is one odd summand
 for each body edge, so $q$ must be even, say $q=2m$.  It follows that
 $B\tH(A)\simeq L(m)$.  The upper end of each spine is by definition
 not a body vertex, but the lower end is a body vertex by the Lemma.
\end{proof}

\begin{remark}\label{rem-rolling}
 The graph $H(A)$ is then obtained from $\tH(A)$ by connecting the two
 ends of $B\tH(A)$ together (``rolling the hedgehog into a ball'').
 This produces a copy of $C(m)$ with attached spines.  If $m\neq 1$ we
 can say that the spines are precisely the edges $e$ such that
 $|p_A^{-1}\{e\}|$ is even.  However, this formulation is incorrect
 when $m=1$ because of the special features of $C(1)$.
\end{remark}

Corollary~\ref{cor-add-edge} allows us to describe the ring $S(H(A))$
as $R(m)\ot\left(\Z[x]/x^2\right)^{\ot k}$, where $m$ is half the
number of body edges, and $k$ is the number of spines.  We can be more
precise about the indexing as follows.  Put $K=A^\#\sm\{0\}$, so the
map $k\mapsto\tp_A(e_k)$ gives a bijection $K\to\edg(\tH(A))$.  Let
$K_0$ be the subset corresponding to spines, and let $K_1$ be the
subset corresponding to body edges.  Note that $K_1$ inherits a total
order from $N_{2n}$ and $|K_1|=2m$.  We find that
$S(H(A))=R(K_1)\ot E(K_0)$, and this has a basis consisting of
monomials $x_J$ where $J\sse K$ and $J\cap K_1$ is a sparse
subset of $K_1$.  When $m\neq 1$ we see that $K$ also bijects with
$\edg(H(A))$.  When $m=1$ we need a few additional words but nothing
significant changes.  The two body edges in $\tH(A)$ get collapsed
down to a single edge in $H(A)$ but this does not matter because
\[ S(C(1)) = R(1) = \frac{\Z[x_1,x_2]}{(x_1+x_2,x_1x_2)}
    = \frac{\Z[x_1]}{x_1}^2 = S(\text{an edge}).
\]

Now put
\[ S' = R(n)/(x_i+x_{i+1}\st i\in A). \]
If $i\in A$ then $p_A(e_i)=p_A(\ov{e}_{i+1})$ and so $x_i+x_{i+1}$
maps to zero in $S(H(A))$.  We therefore have an induced homomorphism
$\tht\:S'\to S(H(A))$, which is surjective because $p_A$ is surjective
on edges.  We can apply this to the relation
$\prod_{i=1}^{2n}(1+tx_i)=1$ in $R(n)[t]$.  For any spine $e$ there
are an even number (say $2k$) of edges $e_i$ that map to $e$, half
with one orientation and half with the other orientation.  The
corresponding terms in the product therefore give
$(1+tx_e)^k(1-tx_e)^k=(1-t^2x_e^2)^k$ but $x_e^2=0$ so this is just
$1$.  Similarly, if $e$ is a positively oriented body edge then the
corresponding terms in the product will be $(1+tx_e)^{k+1}(1-tx_e)^k$
for some $k$, but this just simplifies to $1+tx_e$.  Thus, $\tht$
sends the unique nondegenerate relation in $R(n)=S(C(n))$ to the
unique nondegenerate relation in $S(H(A))$, and it follows that the
map $S'\to S(H(A))$ is an isomorphism as claimed.

Now suppose that for all $m<n$ the map $\phi\:R(m)\to H^*(X(m))$ is an
isomorphism, or equivalently the map $S(C(m))\to H^*(Y(C(m)))$ is an
isomorphism.  Provided that $A\neq\emptyset$, the body of $H(A)$ will
be isomorphic to $C(m)$ for some $m<n$, so the map
$\psi\:S(BH(A))\to H^*(Y(BH(A)))$ is an isomorphism.  After applying
Corollary~\ref{cor-add-edge} once for each spine, we deduce that the
map $\psi\:S(H(A))\to H^*(Y(H(A)))$ is an isomorphism.  We leave it to
the reader to check that nothing goes wrong if there are $0$ or $2$
body edges.

\section{The Mayer-Vietoris spectral sequence}
\label{sec-MVSS}

We will analyse $H^*(X(n))$ using a version of the Mayer-Vietoris
spectral sequence.  It will be convenient to explain this spectral
sequence more abstractly in the present section, and then specialise
to the Khovanov-Springer context in the next section.

Let $X$ be an ordered simplicial complex, with a list of subcomplexes
$X_1,\dotsc,X_m$.  We assume that $|X|=\bigcup_i|X_i|$, or
equivalently that every simplex in $X$ is contained in $X_i$ for some
$i$.  For any subset $I\sse N_m$ we put $X_I=\bigcap_{i\in I}X_i$,
with the convention that $X_\emptyset=X$.  In the case $m=2$ we have a
Mayer-Vietoris sequence
\[ \dotsb \to H^*(X_\emptyset)\to H^*(X_1)\tm H^*(X_2) \to
    H^*(X_{12}) \xra{} H^{*+1}(X_\emptyset) \to \dotsb.
\]
The Mayer-Vietoris spectral sequence will be a generalisation that
works for $m>2$.  It is possible to construct a version that converges
to $H^*(X)$, but we prefer to use a different version where $H^*(X)$
appears as part of the initial page, and the spectral sequence
converges to zero.

Let $E^*$ be the exterior algebra on generators $e_1,\dotsc,e_m$.  We
can also write the elements of $I$ in increasing order as
$\{i_1,\dotsc,i_r\}$ and put $e_I=e_{i_1}\dotsb e_{i_r}$.  We put
$E_I=\Z e_I\leq E^*$ and $HT^{**}=\prod_IE_I\ot H^*(X_I)$.  This is
bigraded, with $E_I\ot H^q(X_I)$ in bidegree $(|I|,q)$. Using the
usual cup product and restriction maps we have pairings $H^*(X_I)\ot
H^*(X_J)\to H^*(X_{I\cup J})$, which we combine with the ring
structure on $E^*$ to make $HT^{**}$ into a bigraded ring.  We also
have an element $u=\sum_ie_i\in E^1=HT^{1,0}$, and it is
straightforward to check that $u^2=0$.

\begin{proposition}\label{prop-MVSS}
 There is a spectral sequence converging to the zero group, with
 $HT^{**}$ as the $E_1$ page and $d_1(a)=au$.
\end{proposition}

Before proving this, we explain how it works out in the case $m=2$.
We label the inclusion maps as follows:
\[ \xymatrix{
 X_{12} \ar[r]^{i_1} \ar[d]_{i_2} & X_1 \ar[d]^{j_1} \\
 X_2 \ar[r]_{j_2} & X.
} \]
We have
\begin{align*}
 E_1^{0*} &= H^*(X) \\
 E_1^{1*} &= H^*(X_1)\tm H^*(X_2) \\
 E_1^{2*} &= H^*(X_{12}).
\end{align*}
The differential $d_1$ is given by $d_1(a)=(j_1^*(a),j_2^*(a))$ on the
first column, and $d_1(a_1,a_2)=i_1^*(a_1)-i_2^*(a_2)$ on the second
column.  The $E_2$ page has $\ker(j_1^*)\cap\ker(j_2^*)$ in the zeroth
column and $\img(i_1^*)+\img(i_2^*)$ in the second column.  The first
column is zero because the usual Mayer-Vietoris sequence is exact in
the middle.  The connecting homomorphism of the usual Mayer-Vietoris
sequence provides an isomorphism between the two remaining columns,
which can be identified with $d_2$, so the $E_3$ page is zero.

\begin{proof}
 Put $CT^{**}=\prod_IE_I\ot C^*(X_I)$ (where $C^*(\cdot)$ refers to
 simplicial cochains).  This is again bigraded, with $E_I\ot C^q(X_I)$
 in bidegree $(|I|,q)$. Using the usual cup product and restriction
 maps we have pairings $C^*(X_I)\ot C^*(X_J)\to C^*(X_{I\cup J})$,
 which we combine with the ring structure on $E^*$ to make $CT^{**}$
 into a bigraded ring.  There are the usual differentials
 $C^q(X_I)\to C^{q+1}(X_I)$, which can be combined in an obvious way
 to give differential $\dl'$ on $CT^{**}$ of bidegree $(0,1)$.  We can
 interpret $u$ as an element of $CT^{1,0}$ and define another
 differential $\dl''$ of bidegree $(1,0)$ by $\dl''(a)=ua$.  The
 differentials $\dl'$ and $\dl''$ anticommute, so the sum
 $\dl=\dl'+\dl''$ is another differential.  Note that $\dl'$ satisfies
 a Leibniz rule $\dl'(ab)=\dl'(a)b\pm a\dl'(b)$, whereas
 $\dl''(ab)=\dl''(a)b$, so $\dl$ does not have any straightforward
 interaction with the product structure.

 As usual, the double complex $CT^{**}$ gives rise to two spectral
 sequences, both converging to $H^*(CT^{**};\dl)$.  For one of them the
 $E_1$ page is $CT^{**}=\prod_IH^*(X_I)e_I$.  This still has a ring
 structure, and the differential $d_1$ is given by multiplication by
 $u$.

 The $E_0$ page of the other spectral sequence splits, as a module
 over $E^*$, into a product indexed by the simplices in $X$.  Fix a
 simplex $s$, and let $I_s$ be the set of indices $i$ such that
 $s\in\text{simp}(X_i)$.  By hypothesis we have $I_s\neq\emptyset$.
 The summand corresponding to $s$ is just the exterior algebra
 $E[e_i\st i\in I_s]$ and the differential is multiplication by the
 image of $u$ in this ring.  It follows that the cohomology for this
 summand is zero, and thus that the $E_1$ page is zero, so both
 spectral sequences converge to zero.
\end{proof}

\begin{corollary}\label{cor-MVSS-iso}
 Suppose we have a first quadrant bigraded abelian group $A^{**}$
 equipped with a differential $d_1$ of bidegree $(1,0)$ such that
 $H^*(A^{**};d_1)=0$.  Suppose we also have a map $f\:A^{**}\to
 HT^{**}$ that is compatible with the bigrading and with $d_1$, such
 that $f\:A^{pq}\to HT^{pq}$ is an isomorphism whenever $p>0$.  Then
 $f\:A^{0*}\to HT^{0*}=H^*(X)$ is also an isomorphism.
\end{corollary}
\begin{proof}
 Suppose we have $a\in HT^{pq}$ with $p>0$ and $d_1(a)=0$ in
 $HT^{p+1,q}$.  As $f$ is an isomorphism for $p>0$ we see that
 there exists $a'\in A^{pq}$ with $f(a')=a$ and $d_1(a')=0$.  As
 $H^*(A^{**};d_1)=0$ this means that there exists $a''\in A^{p-1,q}$
 with $a'=d_1(a'')$ and so $a=d_1(f(a''))$.  This proves that in the
 MVSS we have $E_2^{pq}=0$ for $p>0$.  As the $E_2$ page is
 concentrated in a single column, there can be no further
 differentials.  As the spectral sequence converges to zero, we
 conclude that the $E_2$ page must already be zero.  This means that
 $d_1$ must identify $H^*(X)=HT^{0*}$ with the kernel of
 $d_1\:HT^{1*}\to HT^{2*}$, but also $A^{0*}$ is the kernel of
 $d_1\:A^{1*}\to A^{2*}$ and $f\:A^{p*}\to HT^{p*}$ is iso for $p=1,2$
 so we conclude that $f\:A^{0*}\to HT^{0*}$ must also be iso, as
 claimed.
\end{proof}

\section{The MVSS for Khovanov-Springer varieties}

In this section we assume that $n>1$.  We apply the theory developed
in the previous section to the spaces $Y(C(n))=X(n)$ and the subspaces
$X(n,i)$ for $1\leq i\leq 2n-1$.  These cover $X(n)$ by
Lemma~\ref{lem-Xni-cover}, and they have the required simplicial
structure by Construction~\ref{cons-triangulation}.  We also recall
from Remark~\ref{rem-YHA-intersection} that $\bigcap_{i\in A}X(n,i)$
is the hedgehog folding space $Y(H(A))$.

Now put
\[ TS^{**} = \frac{R(n)\ot E[e_1,\dotsc,e_{2n-1}]}{
              (e_i(x_i+x_{i+1})\st 0<i<2n)}.
\]
As $x_i+x_{i+1}$ maps to zero in $X(n,i)$ there is an evident map
$f\:TS^{**}\to TH^{**}$ of bigraded rings.  We can interpret
$u=\sum_ie_i$ as an element of $TS^{1,0}$ and define $d_1(a)=au$; then
$f$ also respects $d_1$.

As before, for any $A\sse\{1,\dotsc,2n-1\}$ we put
$e_A=\prod_{a\in A}e_a$ and $E_A=\Z e_A$.  It follows from
Proposition~\ref{prop-SHA} that
\[ TS^{**} = \bigoplus_I S(Y(H(A))) \ot E_A \]

At the end of this section we will prove the following:
\begin{proposition}\label{prop-trivial-homology}
 $H^*(TS^{**};d_1)=0$.
\end{proposition}

Assuming this for the moment, we can prove our main theorem:
\begin{theorem}\label{thm-main-bis}
 The map $\phi\:R(n)\to H^*(X(n))$ is an isomorphism for all $n$.
\end{theorem}
\begin{proof}
 We will work by induction on $n$, noting that the cases $n=0$ and
 $n=1$ are easy.  We may thus assume that the maps
 $\phi\:R(m)\to H^*(X(m))$ are isomorphisms for $m<n$.  Using
 Proposition~\ref{prop-SHA} again we deduce that the map
 $f\:TS^{pq}\to TH^{pq}$ is an isomorphism for $p>0$.  It follows by
 Corollary~\ref{cor-MVSS-iso} that the map $f\:TS^{0*}\to TH^{0*}$ is
 also an isomorphism, but this is the same as
 $\phi\:R(n)\to H^*(X(n))$.
\end{proof}

We now start working towards the proof of
Proposition~\ref{prop-trivial-homology}.

\begin{definition}\label{defn-BTS}
 We let $BTS$ denote the set of monomials $x_Je_A$ where
 \begin{itemize}
  \item[(a)] $A\sse N_{2n-1}$
  \item[(b)] $J\sse A^\#\sm\{0\}$
  \item[(c)] $J\cap B$ is a sparse subset of $B$, where
   $B\sse N_{2n}$ is the set of nonzero body vertices in $\tH(A)$.
 \end{itemize}
 It follows from Theorem~\ref{thm-BR-basis} and
 Proposition~\ref{prop-SHA} that $BTS$ is a basis for $TS^{**}$ over
 $\Z$.  We say that $x_Je_A\in BTS$ is \emph{extendable} if there
 exists $a\in N_{2n-1}$ such that $a<\min(A)$ and $x_Je_{A\cup\{a\}}$
 is also in $BTS$.  (We interpret $\min(\emptyset)$ as $2n$, so the
 first condition is automatic if $A=\emptyset$.)  We write $BTS'$ for
 the set of extendable elements of $BTS$, and $BTS''=BTS\sm BTS'$ for
 the set of unextendable elements.  If $x_Je_A$ is extendable we let
 $a$ be the smallest possible index in $N_{2n-1}$ such that
 $x_Je_{A\cup\{a\}}\in BTS$, and put
 $\eta(x_Je_A)=x_Je_{A\cup\{a\}}$.  This defines a map
 $\eta\:BTS'\to BTS$.
\end{definition}

\begin{lemma}\label{lem-eta-cases}
 Suppose that $x_Je_A\in BTS$ and put $p=\min(A)$ (with
 $\min(\emptyset)=2n$ as before).
 \begin{itemize}
  \item[(a)] The set $Q=\{2,\dotsc,2n\}\sm J$ is nonempty, so we can
   define $q=\min(Q)-1$.  Moreover, we have $q\leq p$.
  \item[(b)] If $r\in N_{2n-1}$ and $r\not\in A$ with
   $x_Je_{A\cup\{r\}}\in BTS$ then we must have $r\geq q$.
  \item[(c)] If $q<p$ then $x_Je_A$ is extendable with
   $\eta(x_Je_A)=x_Je_{A\cup\{q\}}$.
  \item[(d)] If $q=p$ then $x_Je_A$ is not extendable.
 \end{itemize}
\end{lemma}
\begin{proof}
 First note that $\{0,\dotsc,p-1\}\sse A^c$ but $p\in A$, so
 $\{0,\dotsc,p\}\sse A^\#$ but $p+1\not\in A^\#$.  We let $B$ and $S$
 denote the sets of indices for body edges and spines in $\tH(A)$.  It
 is clear from the definitions that $\{1,\dotsc,p-1\}\sse B$.

 As $x_Je_A\in BTS$ we must have $A\sse\{1,\dotsc,2n-1\}$ and
 $J\sse A^\#$, so for $j\in J$ we have $j-1\not\in A$.  It follows
 that if $Q=\emptyset$ we must have $A=\emptyset$.  This means that
 the set $B$ of body edges for $\tH(A)$ is all of $N_{2n}$.  Moreover,
 $J\cap B$ must be sparse in $B$, so $2n\not\in J$, which contradicts
 $Q=\emptyset$.  It follows that $Q\neq\emptyset$ after all, so we can
 put $q=\min(Q)-1$, and we find that $\{2,\dotsc,q\}\sse J$ but
 $q+1\not\in J$.  As $x_Je_A\in BTS$ we must have $J\sse A^\#$, and it
 follows that $q\leq p$.  This proves~(a).

 Suppose that $r\not\in A$ and $1\leq r<q$ and we put $A'=A\cup\{r\}$.
 We then have $(A')^\#=A^\#\sm\{r+1\}$ but $r+1\in J$ so
 $J\not\sse(A')^\#$ so $x_Je_{A'}\not\in BTS$.  This proves (the
 contrapositive of) claim~(b), and claim~(d) follows directly.
 Moreover, in~(c) we need only prove that $x_Je_{A\cup\{q\}}\in BTS$,
 because~(b) will show that $q$ is minimal subject to this property.

 We next discuss the case of~(c) where we have the stronger inequality
 $q<p-1$.  We put $A'=A\cup\{q\}$, and write $B'$ and $S'$ for the
 sets of indices of body edges and spines in $\tH(A')$.  Adding $q$ as
 an extra pinch point has the effect of folding the body edges $e_q$
 and $e_{q+1}$ together to make a new spine, so
 $(A')^\#=A^\#\sm\{q+1\}\supseteq J$ and $B'=B\sm\{q,q+1\}$ and
 $S'=S\cup\{q\}$.  The assumption $x_Je_A\in BTS$ means that $J$ is
 sparse in $B$.  We must show that the set $J'=J\cap B'$ is sparse in
 $B'$.  Recall that $q\in J$ but $q+1\not\in J$.  If $j\in J'$ with
 $j<q$ then $|J'_{>j}|=|J_{>j}|-1$ and $|(B'\sm J')_{>j}|=|(B\sm
 J)_{>j}|-1$, whereas for $j>q+1$ we have $|J'_{>j}|=|J_{>j}|$ and
 $|(B'\sm J')_{>j}|=|(B\sm J)_{>j}|$.  The sparsity condition is clear
 from this, so $x_Je_{A'}\in BTS$ as required.

 Finally, we consider the case of~(c) where $q=p-1$, so $p-1\in J$ but
 $p\not\in J$.  Let $m$ be the next element of $A^\#$ after $p$ (or
 $m=2n+1$ if there is no such element).  If $m-p$ is odd then $p\in B$
 and $B'=B\sm\{p-1,p\}$ and the argument is essentially the same as in the
 $q<p-1$ case.  Suppose instead that $m-p$ is even, so $p\in S$.  We
 then have $(A')^\#=A^\#\sm\{p\}\supseteq J$ and $B'=B$ and
 $S'=S\sm\{p\}$.  The sparsity condition is therefore unchanged, and
 again we have $x_Je_{A'} \in BTS$.
\end{proof}
\begin{remark}\label{rem-empty-extendable}
 As a special case, we see that any element of the form
 $x_Je_\emptyset$ in $BTS$ is extendable with
 $\eta(x_Je_\emptyset)=x_Je_q$, because the inequality
 $q<\min(\emptyset)=2n$ holds automatically.
\end{remark}

\begin{lemma}\label{lem-BTS-subset}
 If $x_Je_A\in BTS$ and $A'\sse A$ then $x_Je_{A'}\in BTS$.
\end{lemma}
\begin{proof}
 By an evident inductive reduction, we need only treat the case where
 $|A'|=|A|-1$.  As in Section~\ref{sec-hedgehog}, we write
 \[ A^\# = \{i\in \{0,\dotsc,2n\} \st i-1\not\in A\}
     = \{i_0,\dotsc,i_r\}
 \]
 with $0=i_0<i_1<\dotsb<i_r$ and $i_1=1$, and we put $i_{r+1}=2n+1$.  We then
 put
 \begin{align*}
  E &= A^\#\sm\{0\} = \{i_1,\dotsc,i_r\} \\
  B &= \{i_k\in E\st i_{k+1}-i_k \text{ is odd } \} \\
  S &= \{i_k\in E\st i_{k+1}-i_k \text{ is even } \},
 \end{align*}
 so $B$ bijects with the set of body edges in $\tH(A)$, and $S$
 bijects with the set of spines.  We write $E'$, $B'$ and $S'$ for the
 corresponding sets defined in terms of $A'$.

 By hypothesis we have $J\sse E$, and the set $K=J\cap B$ is sparse in
 $B$.  We must show that $J\sse E'$, and that the set $K'=J\cap B'$ is
 sparse in $B'$.

 As $A'\sse A$ with $|A'|=|A|-1$ we have $E'=E\cup\{p\}$ for some
 $p$.  It is thus clear that $J\sse E'$.  Note also that we must have
 $i_u<p<i_{u+1}$ for some $u\in\{1,\dotsc,r\}$.  For the sparsity
 condition there are four cases to consider.
 \begin{itemize}
  \item[(a)] Suppose that $p-i_u$ and $i_{u+1}-p$ are both even.  Then
   $B'=B$, so $K'=K$ and this is certainly sparse in $B'$.
  \item[(b)] Suppose that $p-i_u$ is even and $i_{u+1}-p$ is odd.  Then
   $B'=(B\sm\{i_u\})\cup\{p\}$, and $p$ has the same position as $i_u$
   relative to the rest of $B$.  Note that $p\not\in E$ so
   $p\not\in J$.  This means that the pair $(B',K')$ of ordered
   sets is isomorphic to $(B,K\sm\{i_u\})$, so $K'$ is again sparse in
   $B'$.
  \item[(c)] Suppose that $p-i_u$ is odd and $i_{u+1}-p$ is even.  Then
   again $B'=B$ and $K'=K$, so the set $K'$ is sparse in $B'$.
  \item[(d)] Suppose that $p-i_u$ and $i_{u+1}-p$ are both odd.  Then
   $B'=B\cup\{i_u,p\}$.  If $i_u\not\in J$ then $K'=K$, and this set
   is sparse in $B$, so it is certainly sparse in the larger set $B'$.
   Suppose instead that $i_u\in J$, so $K'=K\cup\{i_u\}$.  Consider a
   point $k\in K'$.
   \begin{itemize}
    \item[(i)] If $k<i_u$ then $K'_{>k}=K_{>k}\cup\{i_u\}$ and
     $(B'\sm K')_{>k}=(B\sm K)_{>k}\cup\{p\}$.
    \item[(ii)] If $k>i_u$ then $k\geq i_{u+1}>p$ so $K'_{>k}=K_{>k}$
     and $(B'\sm K')_{>k}=(B\sm K)_{>k}$.
    \item[(iii)] Now suppose that $k=i_u$.  If $K_{>k}$ is nonempty,
     we let $q$ be the smallest element of $K_{>k}$.  This will be
     equal to $i_v$ for some $v>u$ and so will be larger than $p$.  We
     will thus have $K'_{>k}=K_{>k}=K_{>q}\cup\{q\}$ and
     \[ (B'\sm K')_{>k}=(B\sm K)_{>k}\cup\{p\}\supseteq
         (B\sm K)_{>q}\cup\{p\}.
     \]
     On the other hand, if $K_{>k}=\emptyset$ we have
     $K'_{>k}=\emptyset$ and $(B'\sm K')_{>k}\supseteq\{p\}$.
   \end{itemize}
   In all of these cases the sparsity condition
   $|K'_{>k}|<|(B'\sm K')_{>k}|$ follows immediately from the assumed
   sparsity condition $|K_{>k}|<|(B\sm K)_{>k}|$. \qedhere
 \end{itemize}
\end{proof}

\begin{proposition}\label{prop-eta-bijection}
 The map $\eta$ gives a bijection $BTS'\to BTS''$.
\end{proposition}
\begin{proof}
 First consider an element $x_Je_A\in BTS'$.  Put
 \[ U = \{u\in N_{2n-1} \st u<\min(A) \text{ and }
           x_Je_{A\cup\{u\}} \in BTS \}.
 \]
 By the definition of $BTS'$ we have $U\neq\emptyset$.  We put
 $q=\min(U)$, so $\eta(x_Je_A)=x_Je_{A\cup\{q\}}$.  We claim that this
 element lies in $BTS''$.  If not, there would exist $r<q$ such that
 $x_Je_{A\cup\{q,r\}}\in BTS$.  By Lemma~\ref{lem-BTS-subset} this
 would give $x_Je_{A\cup\{r\}}\in BTS$, so $r\in U$, contradicting the
 definition $q=\min(U)$.  Thus, $\eta$ at least gives a map
 $BTS'\to BTS''$.

 In the opposite direction, consider an element $x_Je_A\in BTS''$.  By
 Remark~\ref{rem-empty-extendable} we must have $A\neq\emptyset$.  We
 put $p=\min(A)$ and $\zt(x_Je_A)=x_Je_{A\sm\{p\}}$.
 Lemma~\ref{lem-BTS-subset} ensures that this lies in $BTS$, and it is
 clearly extendable by $p$, so this construction gives a map
 $\zt\:BTS''\to BTS'$.  It is visible that $\zt\eta=1\:BTS'\to BTS'$.

 We claim that $\eta\zt\:BTS''\to BTS''$ is also the identity.  To see
 this, we consider again an element $x_Je_A\in BTS''$ and put
 $p=\min(A)$.  As in Lemma~\ref{lem-eta-cases} we also put
 $Q=\{2,\dotsc,2n\}\sm J$ and $q=\min(Q)-1$.  As $x_Je_A$ is not
 extendable, that Lemma tells us that $q=p$.  Now put
 $A'=A\sm\{p\}=A\sm\{q\}$ and $p'=\min(A')>p=q$.  Using
 Lemma~\ref{lem-eta-cases} again we find that
 $\eta(x_Jx_{A'})=x_Jx_{A'\cup\{q\}}=x_Jx_A$ as claimed.
\end{proof}

\begin{lemma}\label{lem-trivial-homology}
 Suppose we have
 \begin{itemize}
  \item[(a)] A free abelian group $A$
  \item[(b)] An endomorphism $d\:A\to A$ with $d^2=0$
  \item[(c)] A basis $\{a_i\}_{i\in I}$ for $A$, for some finite,
   totally ordered set $I$
  \item[(d)] A splitting $I=J\amalg K$
  \item[(e)] A bijection $\eta\:J\to K$ such that
   $d(a_j)=\pm a_{\eta(j)}+\text{ lower terms }$ for all $j\in J$.
 \end{itemize}
 Then the homology group $H(A,d)=\ker(d)/\img(d)$ is trivial.
\end{lemma}
\begin{proof}
 Let $B$ be the subgroup generated by $\{a_j\st j\in J\}$, and let $C$
 be the subgroup generated by $\{a_k\st k\in K\}$, so $A=B\oplus C$.
 We can thus decompose $d$ into four homomorphisms
 \begin{align*}
  p \: & B\to B & q \: & C\to B \\
  r \: & B\to C & s \: & C\to C
 \end{align*}
 such that $d(b,c)=(p(b)+q(c),\;r(b)+s(c))$ for all $b\in B$ and
 $c\in C$.  The condition $d^2=0$ becomes
 \begin{align*}
  p^2+qr &= 0 & pq+qs  &= 0 \\
  rp+sr  &= 0 & rq+s^2 &= 0.
 \end{align*}
 Condition~(e) implies that $r\:B\to C$ is an isomorphism, so we can
 rewrite these relations as $q=-p^2r^{-1}$ and $s=-rpr^{-1}$, so
 \[ d(b,c) = (p(b-pr^{-1}(c)),\;r(b-pr^{-1}(c))). \]
 It follows that if $d(b,c)=0$ we have $b=pr^{-1}(c)$ and so
 $(b,c)=d(r^{-1}(c),0)$ as required.
\end{proof}

\begin{proof}[Proof of Proposition~\ref{prop-trivial-homology}]
 We will apply Lemma~\ref{lem-trivial-homology} with $A=TS^{**}$,
 $I=BTS$, $J=BTS'$ and $K=BTS''$.  We need only introduce a suitable
 ordering and verify condition~(e).  We order subsets of $N_{2n}$
 lexicographically as before, and declare that $x_Je_A<x_Ke_B$ iff
 either $J<K$, or ($J=K$ and $A>B$).  (Note the reversal in the second clause.)

 Consider an element $x_Je_A\in BTS'$, so
 $\eta(x_Je_A)=x_Je_{A\cup\{p\}}$ for some $p<\min(A)$.  Note that
 $d_1(x_Je_A)=\sum_{t=1}^{2n-1}z_t$, where $z_t=x_Je_Ae_t$.  Note that
 $e_Ae_t$ is $0$ (if $t\in A$) or $\pm e_{A\cup\{t\}}$ (if $t\not\in
 A$).  If $t<p$ then (by the definition of $\eta$) the monomial $x_J$
 cannot satisfy the sparsity condition for $x_Je_{A\cup\{t\}}$ to be
 in $BTS$, so Lemma~\ref{lem-BR-spans} implies that
 $x_Je_{A\cup\{t\}}$ can be written as a sum of terms that are lower
 with respect to our ordering on $BTS$.  The same applies to any terms
 where $t>p$ but $x_Je_{A\cup\{t\}}\not\in BTS$.  This just leaves
 terms where $t>p$ and $t\not\in A$ and $x_Je_{A\cup\{t\}}\in BTS$.
 These are lower than the main term $\eta(x_Je_A)=x_Je_{A\cup\{p\}}$,
 by the second clause in our definition of the order.
\end{proof}

\end{document}